\documentclass{amsart}

\usepackage{amsmath,amstext,amssymb,amsthm,amsfonts}
\usepackage[all,cmtip]{xy}
\usepackage{tikz-cd}
\usepackage[normalem]{ulem}

\theoremstyle{plain}
\newtheorem*{theorem*}{Theorem}
\newtheorem*{remark*}{Remark}
\newtheorem*{example*}{Example}
\newtheorem{lemma}{Lemma}[section]
\newtheorem{proposition}[lemma]{Proposition}
\newtheorem{corollary}[lemma]{Corollary}
\newtheorem{theorem}[lemma]{Theorem}

\newtheorem*{conjecture*}{Conjecture}

\newtheorem{prop}[lemma]{Proposition}

\theoremstyle{definition}
\newtheorem{definition}[lemma]{Definition}

\newtheorem{example}[lemma]{Example}

\theoremstyle{remark}
\newtheorem{remark}[lemma]{Remark}

\newtheorem{notation}[lemma]{Notation}

 \newcommand{\paren}[1]{\left({#1}\right)}
\renewcommand{\Im}{\operatorname{Im}}

\newcommand{\bC}{{\mathbb C}}

\newcommand{\bR}{{\mathbb R}}

\newcommand{\abs}[1]{\left|{#1}\right|}

\def\quotient#1#2{%
    \raise1ex\hbox{$#1$}\Big/\lower1ex\hbox{$#2$}%
}
\newcommand{\floor}[1]{\lfloor #1 \rfloor}

\begin{document}

\date{July 28, 2021.}
\setcounter{tocdepth}{1}
\title[On Schwartz Equivalence of Quasidiscs and Other Planar Domains]{On Schwartz Equivalence of Quasidiscs and Other Planar Domains}
 \author{Eden Prywes and Ary Shaviv}
\address{Department of Mathematics, Princeton University, Princeton, New Jersey 08544}
\email{eprywes@princeton.edu; ashaviv@math.princeton.edu}
\subjclass[2010]{Primary 46A11, Secondary 30C62}

\maketitle

\begin{abstract}
Two open subsets of $\bR^n$ are called Schwartz equivalent if there exists a diffeomorphism between them that induces an isomorphism of Fr\'echet spaces between their spaces of Schwartz functions. In this paper we use tools from quasiconformal geometry in order to prove the Schwartz equivalence of a few families of planar domains. We prove that all quasidiscs are Schwartz equivalent and that any two non-simply-connected planar domains whose boundaries are quasicircles are Schwartz equivalent.  We classify the two Schwartz equivalence classes of domains that consist of the entire plane minus a quasiarc and prove a Koebe-type theorem, stating that any planar domain whose connected components of its boundary are finitely many quasicircles is Schwartz equivalent to a circle domain. We also prove that the notion of Schwartz equivalence is strictly finer than the notion of $C^\infty$-diffeomorphism by constructing examples of open subsets of $\bR^n$ that are $C^\infty$-diffeomorphic and are not Schwartz equivalent.
\end{abstract}

\tableofcontents

\section{Introduction} A real valued $C^\infty$-smooth function on an open subset of $\bR^n$ is called a Schwartz function, if it and all of its partial derivatives rapidly decay when approaching any boundary point of the subset, including $\infty$ if the subset is unbounded.
The space of all Schwartz functions on a given subset $U\subset\bR^n$ is a Fr\'echet space denoted by $\mathcal{S}(U)$, and is called the Schwartz space of $U$.
Schwartz spaces were first introduced on $\bR^n$ by Laurent Schwartz (see \cite{schwartz}) and throughout the years were defined and studied in various contexts on various objects, e.g., semi-simple/reductive Lie groups \cite{harishchandra,arthur,arthur2,casselman,casselmanhechtmili}, Nash ($C^\infty$-smooth semi-algebraic) manifolds \cite{dC,AG},
smooth semi-algebraic stacks \cite{sakellaridis}, (possibly singular) algebraic varieties \cite{ES} and $C^\infty$-smooth manifolds definable in polynomially bounded o-minimal structures \cite{Shaviv}. First introduced in the first half of the $20^\text{th}$ century, Schwartz spaces still play an important role in many fields of mathematics, such as Harmonic Analysis, Representation Theory \cite{GSS} and Number Theory \cite{Ge,GHL,CG}.

\subsection*{Historic motivation.} Originally, Schwartz spaces arose in Functional Analysis, in the context of the Fourier transform. Let us briefly recall the simplest application. Consider the Fourier transform on the real line, given by the formula 
 \[
 \mathcal{F}(f)(\omega):=\int\limits_{-\infty}^{\infty}f(x)e^{-2\pi i x \omega}dx.
 \]
 A straightforward calculation shows that 
 \[
 \mathcal{F}(\frac{d}{dx}f(x))(\omega)=2\pi i \omega\mathcal{F}(f(x))(\omega)
 \]
 and
 \[
 \mathcal{F}(x
 f(x))(\omega)=\frac{-1}{2\pi i} \frac{d}{d\omega}\mathcal{F}(f(x))(\omega).
 \]
 
Intuitively, these relations imply that if $\mathcal{F}(f)(\omega)$ is differentiable $k$ times, then $f(x)$ is integrable even after being multiplied by the monomial $x^k$. So $f$ decays at $\infty$ faster than $x^{-k}$, and vise-versa. We thus think of the Fourier transform as interchanging the property of being ``differentiable to a high order" and the property of ``fast decay at $\infty$". Rigorously, one can show that the Fourier transform given by the formula above is an automorphism on the space of (complex valued) Schwartz functions on the real line.

A tempered distribution on the real line is a continuous linear functional on the Schwartz space of the real line. By duality the Fourier transform is defined on the space of tempered distributions. If $\xi$ is a tempered distribution, then its Fourier transform is defined by $\mathcal{F}(\xi)(s):=\xi(\mathcal{F}(s))$, for any Schwartz function $s$. 
Many functions naturally define tempered distributions by integration, i.e., for any Schwartz function $s$, $(f,s):=\int\limits_{-\infty}^{\infty}f(x)s(x)dx$, whenever this integral makes sense. In particular, Schwartz functions, polynomials, compactly supported continuous functions and trigonometric polynomials define tempered distributions. Not all tempered distributions arise by this integration process, e.g., the Dirac delta distribution at the origin defined by $(\delta_0,s):=s(0)$, for any Schwartz function $s$. Having said that, Schwartz's theorem asserts that every tempered distribution is a derivative  of finite order (in the distributional sense) of some continuous function of polynomial growth (see \cite[Theorem 8.3.1]{Fr}). This approach gives a wide rigorous framework to understand the notion of conjugate variables in Quantum Mechanics. The Fourier transform of the Dirac delta distribution, for instance, is a constant function. 
\subsection*{The problem of Schwartz equivalence.}Two open subsets of $\bR^n$ are called Schwartz equivalent if there exists a $C^\infty$-diffeomorphism between them that induces, by composition, an isomorphism of Fr\'echet spaces between their Schwartz spaces.
Not every $C^\infty$-diffeomorphism of open subsets of $\bR^n$ induces such an isomorphism, as illustrated by the following example.

\begin{example}\label{example-of-all-evil}
The map $\operatorname{exp}\colon\mathbb{R}\to \bR_{>0}$ is a $C^\infty$-diffeomorphism whose inverse is the natural logarithm.
Take a $C^\infty$-smooth function $f\colon\mathbb{R}\to\bR$ such that $f(x)=0$
for any $x<0$ and $f(x)=e^{-x}$ for any $x>1$.
Such an $f$ clearly exists and $f\in\mathcal{S}(\mathbb{R})$.
Then, for any $y>e$, we have $$f(\log(y))=e^{-\log(y)}=\frac{1}{y},$$ and so $f\circ \log\notin\mathcal{S}(\bR_{>0})$. 
\end{example}

Thus, a natural problem is to determine under what conditions two open subsets $U,V\subset \bR^n$ are Schwartz equivalent.  A necessary condition is that $U$ and $V$ are isomorphic as $C^\infty$-manifolds. It was not clear whether this condition is also sufficient and in this paper we prove it is not. In fact, the notion of Schwartz equivalence is implicitly used in most of the theories mentioned above, e.g., the space of tempered distributions on a Nash manifold is well defined and so may be studied, only due to the fact that any semi-algebraic $C^\infty$-diffeomorphism induces a Schwartz equivalence. For other instances in which the question of Schwartz equivalence is implicitly addressed see \cite[p. 18]{CG},\cite[Corollary 4.9, Lemma 6.1]{Shaviv} and \cite[Lemmas 3.6(i), 5.1]{ES}.

In the main part of this paper we use tools from quasiconformal geometry, where the main objects of study are quasiconformal and quasisymmetric maps, to prove the Schwartz equivalence of a few families of open subsets of the plane (we naturally identify $\bC$ with $\bR^2$).
Quasisymmetries are generalizations of conformal maps on $\bC$.  Recall that a conformal map is a diffeomorphism that infinitesimally maps circles to circles.
Quasisymmetries (on $\bC$ these are the same as quasiconformal maps) can be defined as maps that infinitesimally send circles to ellipses that have bounded eccentricity.
While the only conformal maps from $\bC$ to $\bC$ are affine linear maps, there are many more quasisymmetries.
In fact, any bi-Lipschitz map is a quasisymmetry.

The classes of domains that we will study below will often be defined as having boundaries that are quasicircles.
A quasicircle is a closed subset of $\bC$ that is an  image of either the unit circle or $\bR$ under a quasisymmetric map from $\bC$ to itself. Equivalently, one may define a quasicircle as a closed set in $\bC$ whose closure inside $\widehat{\mathbb{C}}$ is an image of the unit circle under a quasisymmetric map from $\widehat{\mathbb{C}}$ to itself. In this paper we prefer the former approach and so whenever possible we will work in the complex plane rather than in the Riemann sphere. A quasidisc is a simply-connected domain in $\bC$ whose boundary is a quasicircle.
In \cite{ahlfors63}, Ahlfors gave an entirely geometric condition that characterizes the planar topological circles that are quasicircles.
This, in particular, showed that bounded quasicircles coincide with planar topological circles that have no zero-angle cusps.
He also showed that a conformal map between two simply-connected domains whose boundaries are quasicircles extends to a quasisymmetry of $\bC$ (we will use this fact often).
This contrasts sharply with the conformal case where if the boundary of a domain is not analytic, then a conformal map from the domain to a disc cannot be extended to a conformal map beyond the boundary.

We do not attempt to give a thorough introduction to the theory of quasiconformal maps.
In Section \ref{chapter-prelim}, we provide all the necessary definitions, most importantly, that of a quasisymmetry, a quasicircle and a quasidisc.  We also review all the results from the theory that we will use.  For references regarding quasiconformal maps and quasicircles see \cite{lehtovirtanen, ahlfors,AIM} and \cite{GH}. 

In the second part of this paper we explore examples outside the quasiconformal realm. In particular we show the Schwartz equivalence of some exponential cusp domains to the unit disc. We also give sufficient conditions for countable unions of intervals in the real line to be Schwartz equivalent and sufficient conditions for countable unions of intervals in the real line to be \emph{non}-Schwartz equivalent. By doing so we prove that the Schwartz equivalence relation is strictly finer than the $C^\infty$-diffeomorphic equivalence relation, already in dimension 1. In the plane we construct an example of \emph{simply-connected} sets that are not Schwartz equivalent, despite being $C^\infty$-diffeomorphic by the Riemann mapping theorem.  

\subsection*{ Main results.} 
The main results of this paper are as follows:
\begin{enumerate}
\item Any two quasidiscs are Schwartz equivalent (Theorem \ref{theorem-quasidiscs-are-schwartz-equivalent}).
\item If $U\subset\bC$ is an unbounded domain whose boundary is a bounded quasicircle, then $U$ is Schwartz equivalent to $\bC\setminus\overline {\mathbb D}$ (Theorem \ref{thm-unbouded-domains-whose-boundaries-are-bounded-quasicircles}).
\item If $U\subset \bC$ is a domain whose boundary is a quasiarc, then $U$ is Schwartz equivalent to $\mathbb D$ when $\partial U$ is unbounded (Theorem \ref{thm-infinite-slit-domains}) and to $\bC\setminus \overline{\mathbb D}$ when $\partial U$ is bounded (Theorem \ref{thm-slit-domains}).
\item If $U \subset \bC$ is a domain whose boundary consists of connected components that are at most countably many points and finitely many quasicircles, out of which at most one quasicircle is unbounded, then $U$ is Schwartz equivalent to a circle domain (Theorem \ref{schwartzkoebe}).
\item There exist pairs of open subsets of $\bR^n$ that are $C^\infty$-diffeomorphic but not Schwartz equivalent. This holds already in $\bR^1$ (Example \ref{example_countable_union_non_equiv}), and in $\bR^2$ such simply-connected pairs exist (Section \ref{nazarov}).  
\end{enumerate}

In \cite[Appendix C]{Shaviv}, the
second author conjectured that $C^\infty$-diffeomorphic open subsets are not necessarily Schwartz equivalent.
More precisely, the author conjectured that if $U,V\subset\bR^n$ are Schwartz equivalent and the Hausdorff dimension of $\partial U$ is strictly greater than $n-1$, then the Hausdorff dimension of $\partial V$ equals the Hausdorff dimension of $\partial U$.
The motivation for this conjecture was the following theorem.
\begin{theorem}[{\cite[Theorem 7.3]{Shaviv}}] For any open $U\subset\bR^n$ there exists $f\in\mathcal{S}(U)$
such that $f(x)>0$ for any $x\in U$.
\end{theorem}
Thus, the Schwartz space should ``detect" the boundary of the subset.
We disprove this more precise conjecture for any $n\geq1$. For any $n\geq2$ we have a simply-connected example: for $n=2$ the unit disc $\mathbb D$
and the bounded planar domain
defined by the Koch snowflake, denoted $U_{\text{Koch}}$ (see Example \ref{examples} (2) below), are quasidiscs and so Schwartz equivalent by Theorem \ref{theorem-quasidiscs-are-schwartz-equivalent}. The boundary of the first is
the unit circle and so has Hausdorff dimension 1, while the boundary of the latter
is the the Koch snowflake, that has Hausdorff dimension $\frac{\log
4}{\log 3}$.  In fact, for any number $d$ so that $1 \le d < 2$, there exists a quasidisc whose boundary has Hausdorff dimension $d$.
 For $n\geq3$ one can easily show that 
 \[
 \mathcal{S}(\mathbb D\times (0,1)^{n-2})\cong\mathcal{S}(U_{\text{Koch}}\times (0,1)^{n-2})
 \]
 using the fact that $\mathcal{S}(\mathbb D)\cong\mathcal{S}(U_{\text{Koch}})$.
The boundaries of these two sets have different Hausdorff dimensions. 
For $n=1$, we have (by Example \ref{example_cantor}) that the complement to the standard $\frac{1}{3}$-Cantor set in $[0,1]$ is Schwartz equivalent to $\bigcup\limits_{n\in\mathbb{N}}(n,n+1)$. The boundary of the latter is countable and so has Hausdorff dimension 0, while the boundary of the first is the standard $\frac{1}{3}$-Cantor set, which has Hausdorff dimension $\frac{\log 2}{\log 3}$.

However, we also prove that a weaker version of this conjecture does  hold, namely that $C^\infty$-diffeomorphic open subsets are
not necessarily Schwartz equivalent. We provide two explicit examples of such sets, one in $\bR^1$ (see Example \ref{example_countable_union_non_equiv}) and the other in $\bR^2$ (see Section \ref{nazarov}). Similar examples may be constructed in higher dimensions.

\subsection*{Bi-H\"older domains.} A key property that is used in proving Theorem \ref{theorem-quasidiscs-are-schwartz-equivalent} is the fact that any bounded quasidisc is a bi-H\"older domain, i.e., it is conformally equivalent to the unit disc via a bi-H\"older map. The first main claim we prove is that any bi-H\"older conformal map between bounded domains induces a Schwartz equivalence (Theorem \ref{conformalbiholderbetweenbounded}). Thus, in particular any two bi-H\"older domains are Schwartz equivalent.

Bi-H\"older domains were characterized in a work of N\"akki and Palka \cite{NP}.
We omit the condition and refer readers to the paper.  However, it is interesting
to note that the condition is entirely geometric and does not depend on the
choice of the conformal map.

It is also important to point out that bi-H\"older domains are not the same
as H\"older domains.
A H\"older domain is a domain in $\bC$ where the conformal map from the disc onto it
is H\"older continuous (see \cite[p. 92]{Po}).
For example, the image of $\mathbb D$ by $e^{\pi z}$ has that $\phi(z) = e^{\pi z}$ is H\"older continuous
but $\phi^{-1}(z) =\frac{1}{\pi}\operatorname{Log}(z) $ is not (for a discussion see \cite[Example 1]{NP}). 
Conversely, the map of the unit disc to any simply-connected domain with
an inward, zero-angle cusp will not be H\"older continuous, while the inverse
will satisfy a H\"older continuity condition near the cusp.

\subsection*{Definable domains.} A different approach to prove the Schwartz equivalence of open subsets was implemented in \cite{Shaviv} and involves tools
from model theory. Let us briefly recall the main idea. Let $\mathcal{R}$ be a polynomially bounded o-minimal structure (see \cite[Subsections 2.1 and 2.2]{Shaviv}) and let $U,V\subset\bR^n$ be two open subsets.
Assume $\phi\colon U\to V$ is a $C^\infty$-diffeomorphism and that moreover $\phi$ is
definable in $\mathcal{R}$. In particular, both $U$ and
$V$ are also definable in $\mathcal{R}$. Then Corollary 4.9 in \cite{Shaviv}
implies that $\phi^*|_{\mathcal{S}(V)}\colon\mathcal{S}(V)\to\mathcal{S}(U)$ is
an isomorphism of Fr\'echet spaces and so $U$ and $V$ are Schwartz equivalent.

The simplest case of a polynomially bounded o-minimal structure is the semi-algebraic
category.
Loosely speaking, sets (resp. maps) definable in this structure
are those sets (maps) that may be described using finitely many polynomial
equations and inequalities.
In the special case of $n=2$, two open semi-algebraic
subsets of $\bR^2$ are $C^\infty$-diffeomorphic if and only if there exists
a semi-algebraic $C^\infty$-diffeomorphism between them.
This follows immediately from \cite[Corollary 3]{Shi1} (see also \cite[Remark B.1.1]{Shaviv}).
This, together with the Riemann mapping theorem, implies that any semi-algebraic, simply-connected, proper, open subset of $\bR^2$ is Schwartz equivalent to the unit disc.
Moreover, it is well known that there is semi-algebraic $C^\infty$-diffeomorphism from $\bR^2$ to the unit disc and therefore any two semi-algebraic simply-connected subsets of $\bR^2$ are Schwartz equivalent. 
Additionally, this approach, together with Koebe's theorem (see Theorem \ref{thm-koebe} below), shows that any open semi-algebraic subset of $\bR^2$ is Schwartz equivalent to a circle domain (an open connected subset of the plane whose connected components of its boundary are all points or circles). 

It should be stressed that neither one of the two approaches described, quasiconformal geometric and model theoretic, is stronger than the other.
The bounded domain defined by the Koch snowflake is not definable in any o-minimal structure but is a quasidisc. On the other hand, the set $$\{(x,y)\in\bR^2|x^2>y^3\}$$
(an outwards facing algebraic cusp) is neither a quasidisc nor a bi-H\"older domain. However, it is semi-algebraic.

The infinite strip $U = \{z \in \bC
: 0<\Im(z)<\pi\}$, even though it is not a Jordan
domain, serves as a good example of the failure of the Riemann map to preserve
the Schwartz space (see Example \ref{examples} (5)).  The set $U$ is semi-algebraic and simply-connected and so $\mathcal{S}(U) \cong \mathcal{S}(\mathbb D)$.  In this case, one can
explicitly compute the Riemann map.  The set U is mapped to the upper half
plane by $e^{z}$, which in turn is mapped to the unit disc by a M\"obius transformation.
As $e^z$ does not preserve the Schwartz space (compare to Example \ref{example-of-all-evil}) and the M\"obius transformation does, their composition fails to preserve the Schwartz space.

\subsection*{Other domains.}The  quasiconformal
geometric and model theoretic approaches do not exhaust all interesting cases of Schwartz equivalent sets. For instance, a special case of our result in Section \ref{cusp_domains} is  that the unit disc is Schwartz equivalent to the set  $$\{(x,y) \in\mathbb{R}^2 | 0 < x < \infty, 0 < y < e^{-x}\}$$ (a one-sided infinite strip with an exponential cusp at infinity). This set is  neither a quasidisc nor a bi-H\"older
domain and it is not definable in any polynomially bounded o-minimal structure.

\subsection*{Structure of this paper.} In Section \ref{chapter-prelim}, we collect all the preliminary definitions and results that will be used in what follows.  These mainly are from the theory of Schwartz functions and from quasiconformal geometry. 

Section \ref{section-on-bounded-quasidiscs} is devoted to proving that any two bounded quasidiscs are Schwartz equivalent. This is achieved by showing that conformal bi-H\"older maps
between bounded domains always induce isomorphisms of Fr\'echet spaces between
the corresponding Schwartz spaces and the fact that the Riemann map from a bounded quasidisc to the unit disc is always such a map. 

In Section \ref{section-on-unbounded-quasidiscs}, we study unbounded domains whose boundaries are quasicircles. These come in two types, simply-connected and non-simply-connected.
A simply-connected, unbounded domain whose boundary is a quasicircle is a quasidisc and we prove it is always Schwartz equivalent to any bounded quasidisc. This is done by showing that M\"obius transformations always induce Schwartz equivalence, and finishes the proof of Theorem \ref{theorem-quasidiscs-are-schwartz-equivalent}.
The second type consists of non-simply-connected, unbounded domains whose boundaries are quasicircles, and we prove that any two such sets are Schwartz equivalent (Theorem \ref{thm-unbouded-domains-whose-boundaries-are-bounded-quasicircles}).

In Section \ref{chapter-complements-to-a-quasiarc}, we study sets whose boundaries are quasiarcs.
We show that any such set is Schwartz equivalent to either the unit disc (in case it is simply-connected, i.e., its boundary is unbounded), or to $\bC\setminus\overline{\mathbb{D}}$ (in case it is not simply-connected, i.e., its boundary is bounded). 

In Section \ref{section-koebeschwartz}, we extend the previous results to domains whose boundaries are composed of finitely many quasidiscs and countably many points.

In Section \ref{nazarov}, we construct two simply-connected planar domains that are $C^\infty$-diffeomorphic but are not Schwartz equivalent.

In Section \ref{cusp_domains}, we prove the Schwartz equivalence of some exponential
cusp domains to the unit disc.

Finally, in Section \ref{line_examples}, we present some examples of Schwartz spaces of subsets of the real line that consist of countable unions of intervals. 

\subsection*{Notation}
Most of the notation in the paper is standard.  We set $\mathbb{N}_0:=\mathbb{N}\cup\{0\}$, i.e., $\mathbb{N}_0$ is the set of all non-negative integers.

By the notion smooth we always mean $C^\infty$-smooth. If $U\subset\bR^n$ is an open subset, $f\colon U\to\bR$ is a smooth function and $k_1,\dots,k_n\in\mathbb{N}_0$, we use multi-index notation for derivatives. So if $k:=(k_1,k_2,\dots,k_n)$, then $\abs{k}:=\sum\limits_{i=1}^{n}k_i$ and
\[
f^{(k)}=\frac{\partial^{|k|}f}{\partial x_1^{k_1}\partial x_2^{k_2}\cdots\partial x_n^{k_n}},
\]
when $\abs{k}\neq0$ and $f^{(k)}=f$ when $\abs{k}=0$. 
If $f:=(f_1,\dots,f_m)\colon U\to\bR^m$ is a smooth function and $k\in(\mathbb{N}_0)^n$ is a multi-index, then we set $f^{(k)}:=(f_1^{(k)},\cdots,f_m^{(k)})$, where $f^{(k)}_i$ denotes the partial derivative of the $i^{\text{th}}$ component function.
For any $x\in U$ we denote $x^k:=\prod\limits_{i=1}^{n}x_i^{k_i}$, where $x_i$ is the $i^{\text{th}}$ coordinate of $x$.

We denote the standard Euclidean norm of a point $x \in \bR^n$ or $\bC$ as $|x|$.  For any two subsets $Z_1,Z_2 \subset \bR^n$ or $\bC$, we set $\operatorname{dist}(Z_1,Z_2):=\inf\limits_{z_1\in Z_1,z_2\in Z_2}\abs{z_1-z_2}$.
For $Z \subset \bR^n$ or $\bC$, we set $\operatorname{dist}(x,Z):=\operatorname{dist}(\{x\},Z)$ and $\operatorname{diam}(Z):=\sup\limits_{x,y\in Z}|x-y|$.

For $x \in \bR^n$ or $\bC$, we set $B(x,R)$ to be the open ball centered at $x$ of radius $R>0$. We denote $\mathbb{D} = B(0,1)$ to be the standard open unit disc either in $\bC$ or in $\bR^2$ (the embedding space will be clear from the context).
By $\widehat \bC$ we denote the extended complex plane (its one point rational compactification, i.e., the Riemann sphere).
A domain in $\bC$ is an open and connected set.
We often treat functions from the complex plane to itself as if they were functions from $\bR^2$ to itself.

When $X$ is any set and $Y\subset X$ is any subset, we denote by $\operatorname{Ext}_Y^X$ the ``extension by zero" operator that takes a real valued function on $Y$ and returns a real valued function on $X$ as follows. For any $f\colon Y\to\mathbb{R}$,
\[   
\operatorname{Ext}_Y^X(f)(x) := 
     \begin{cases}
       f(x) &\quad\text{if } x \in Y, \\
       0 &\quad\text{if } x \in X \setminus Y .\\

     \end{cases}
\]

\subsection*{Acknowledgments.}  E. P. would like to thank Mario Bonk for discussions on the paper. A. S. is grateful to Alexandre Eremenko, Bo'az Klartag, and Dmitry Novikov for valuable discussions related to the subject this paper addresses, and to Charles Fefferman for his great support throughout the past 2 years, and for many interesting discussions. E. P. and A. S. are thankful to Fedor Nazarov and Mikhail Sodin for bringing the example presented in Section \ref{nazarov} to their attention. E. P. was supported by
the Simons Foundation Algorithms and Geometry Collaboration. A. S. was supported by AFOSR Grant FA9550-18-1-069.

\newpage

\section{Preliminaries}\label{chapter-prelim}

\subsection*{Schwartz functions}\label{prelim-schwartz}\begin{definition}A \textit{Schwartz function} on $\bR^n$ is a smooth
function $f\colon\bR^n\to\bR$ such that for any two multi-indices $k,l\in(\mathbb{N}_0)^n$,
$\sup\limits_{x\in\bR^n}\abs{x^l f^{(k)}(x)}<\infty$. The space
of all Schwartz functions on $\bR^n$ is denoted by $\mathcal{S}(\bR^n)$.\end{definition}

\begin{proposition}[e.g., {\cite[Corollary 4.1.2]{AG}}]

$\mathcal{S}(\bR^n)$
has a natural structure of a Fr\'echet space (a metrizable, complete locally
convex topological vector space), where the topology is given
by the family of semi-norms indexed by $(\mathbb{N}_0)^n\times(\mathbb{N}_0)^n$,
$$\|f\|_{k,l}:=\sup\limits_{x\in\bR^n}\abs{x^l f^{(k)}(x)}.$$

\end{proposition}

\begin{definition}

Let $U\subset\bR^n$ be an open subset and let $z\in
U$ be some point. We say that a smooth function $f\colon U\to\bR$ is \textit{flat at $z$}
if $T_z(f)$, its Taylor series at $z$, is identically zero. If $Z\subset U$ is any subset,
we say that $f$ is \textit{flat} in $Z$ if for all $z \in Z$ it is flat at $z$.

\end{definition}

\begin{definition}\label{def-schwartz-on-open}

Let $U\subset\bR^n$ be an open subset. Define the \textit{space of Schwartz
functions} on $U$, $$\mathcal{S}(U):=\bigcap\limits_{z\in \bR^n\setminus
U}\bigcap\limits_{k\in(\mathbb{N}_0)^n}\{f\in\mathcal{S}(\bR^n)|f^{(k)}(z)=0\}.$$
\end{definition}
As a closed subspace of a Fr\'echet space, $\mathcal{S}(U)$ is a Fr\'echet space with the induced topology.
Note that there is a natural bijection between $\mathcal{S}(U)$ and the set
$$\{f\colon U\to\bR|\operatorname{Ext}_U^{\bR^n}(f)\in\mathcal{S}(\bR^n) \text{ and } \operatorname{Ext}_U^{\bR^n}(f)\text{ is flat in }\bR^n\setminus U\}.$$
Thus, we will consider Schwartz functions on $U$ as a class of smooth real-valued functions on $U$.
In this point of view the topology is given by the family
of semi-norms indexed by $(\mathbb{N}_0)^n \times (\mathbb{N}_0)^n$,
$$\|f\|_{k,l}:=\sup\limits_{x\in U}\abs{x^l f^{(k)}(x)}.$$
A partial derivative of a Schwartz function is clearly a Schwartz function as well.

\begin{notation}
If $U\subset\bC$ is an open subset, then by $\mathcal{S}(U)$ we mean the space of Schwartz functions on $U$ when we consider $U$ as an open subset of $\bR^2$.
\end{notation}
The following proposition gives an alternate criteria for a function to be Schwartz.
\begin{proposition}[{\cite[Proposition 3.2.2]{Shaviv}}]\label{prop-characterization-of-schwartz-functions}
Let $U\subsetneq\bR^n$ be an open subset and  $f\in
\mathcal{S}(\bR^n)$. If $f|_{\bR^n\setminus U}\equiv0$, then the following
are equivalent:
\begin{enumerate}
  \item $f\in\mathcal{S}(U)$, i.e., $T_{x_0}(f)\equiv 0$ for any $x_0\in\bR^n\setminus
U$;
  \item for any $m\in\mathbb{N}_0$, $$\sup\limits_{x\in U}\abs{\frac{f(x)}{\operatorname{dist}(x,\bR^n\setminus
U)^m}}<\infty.$$
\end{enumerate}

\end{proposition}

\begin{definition}

Two open subsets $U,V\subset\bR^n$ are called \textit{Schwartz equivalent} if there exists a smooth diffeomorphism $\phi\colon U\to V$ such that $\phi^*|_{\mathcal{S}(V)}:\mathcal{S}(V)\to\mathcal{S}(U)$ is an isomorphism of Fr\'echet spaces. 

\end{definition}

\begin{lemma}[c.f. {\cite[Lemma 4.8]{Shaviv}}]\label{lemma-enough-to-pull-functions-to-get-F-iso}
Let $U,V\subset\bR^n$ be open
subsets and $\phi\colon U\to V$ some map. If $\phi^*(\mathcal{S}(V))\subset\mathcal{S}(U)$,
then $\phi^*\colon\mathcal{S}(V)\to\mathcal{S}(U)$ is continuous. In particular, if in addition $\phi$ is invertible and $(\phi^{-1})^*(\mathcal{S}(U))\subset\mathcal{S}(V)$, then $U$ and $V$ are Schwartz equivalent.

\end{lemma}

\subsection*{Quasidiscs and H\"older maps}\label{prelim-quasiconformal-geometry}
\begin{definition}\label{definitionquasistuff}
Let $J^0\subset\bC$ be a closed subset and denote by $J$ its closure inside $\widehat\bC$. We say that $J^0$ is a \textit{quasicircle} if $J$ is a Jordan curve and there exists $C>0$ such that
\begin{align}\label{eq-quasicirclecondition}
    \operatorname{diam}(J^0(x,y))\leq C|x-y|,
\end{align}
for all $x,y\in J^{0}$. Here, $\operatorname{diam}(J^0(x,y))$ is the smallest diameter of a connected component of $J^0\setminus\{x,y\}$.
A \textit{quasiarc} is a subset of $\bC$ that is a subset of some quasicircle, and such that its closure inside $\widehat{\bC}$ is homeomorphic to $[0,1]$.
A simply-connected open subset of the complex plane is called a \textit{quasidisc} if its boundary is a quasicircle.

\end{definition}

\begin{remark}\label{remark-quasidisc-cannot-be-dense}
We will often use the fact that a quasidisc is never dense in $\bC$.  This follows immediately from the definition.
\end{remark}
We now mention some properties of quasicircles and quasidiscs.  We only mention properties that we will use below.  For a more complete survey we refer readers to \cite{GH}.

\begin{definition}
Let $U,V \subset \bC$ be connected sets. A \textit{quasisymmetric map} $\phi \colon U \to V$ is a homeomorphism for which there exists a homeomorphism $\eta \colon [0,\infty) \to [0,\infty)$ such that for all distinct $x,y,z \in U$
\begin{align*}
    \frac{|\phi(x)-\phi(y)|}{|\phi(x)-\phi(z)|} \le \eta\biggl ( \frac{|x-y|}{|x-z|} \biggr).
\end{align*}
\end{definition}
Note that the inverse of a quasisymmetry is a quasisymmetry, with the function $\sigma(t) = 1/\eta^{-1}(1/t)$.

The following proposition is a well-known fact regarding quasicircles.  We will not use it in what follows but we mention it here for the sake of completeness.
For a proof, compare with Theorem 13.3.1 in \cite{AIM}.
\begin{prop}\label{prop-quasidisc-is-preserved-under-conformal}
Let $J^0$ be a subset of $\bC$. $J^0$ is a quasicircle if and only if $J^0$ is the image of either the unit circle or $\bR$ under a quasisymmetric map from $\bC$ to itself.
\end{prop}

\begin{example}\label{examples}

\begin{enumerate}
\item The upper half plane is a quasidisc.
\item The Koch snowflake with a constant angle in each iteration is a quasicircle (see \cite[Figure 0.2]{Fa} and \cite[Exercise 5.4.1]{Po}).
\item The algebraic cusp $\{z\in\bC|\operatorname{Im}(z)^2=\operatorname{Re}(z)^3\}$ is not a quasicircle.
\item The exponential cusp ($\{z\in\bC|\operatorname{Im}(z)=\pm e^{\frac{-1}{\operatorname{Re}(z)^2}}, \operatorname{Re}(z) \ge 0\}$ ) is not a quasicircle.
\item The infinite open strip $\{z\in\bC|0<\operatorname{Im}(z)<\pi\}$ is not a quasidisc.
\item The set $\{z\in\bC|\operatorname{Im}(z)=\sin(\operatorname{Re}(z))\}$ is a quasicircle.
\item The set $\{z\in\bC|\operatorname{Im}(z)=\sin(\operatorname{Re}(z)^2)\}$ is not a quasicircle.
\item The set $\{z\in\bC|\operatorname{Im}(z)=0,\operatorname{Re}(z)\notin(0,1)\}$ is a quasiarc. 

\end{enumerate}

\end{example}

\begin{definition}
Let $U$ and $V$ be subsets of $\bR^n\text{ }(\text{or }\bC)$.  A map $\phi\colon U\to V$ is called a \textit{H\"older map} if there exist $C,\alpha\in\bR_{>0}$
such that $\abs{\phi(x)-\phi(y)}\leq C\abs{x-y}^\alpha$, for all $x,y\in U$.
A \textit{bi-H\"older map} is an invertible map such that both it and its inverse are H\"older maps.
\end{definition}

\begin{proposition}\label{prop-on-bi-holo-bi-holder-for-quasidiscs}
Let $U\subset\bC$ be a bounded quasidisc. Then, any conformal map $\phi\colon\mathbb{D}\to U$ extends to a bi-H\"older map $\phi \colon \bC \to \bC$, and for any bounded quasidisc such a map exists.
\end{proposition}
This proposition also holds true for unbounded quasidiscs if we replace $\mathbb D$ with the upper half plane.
Proposition \ref{prop-on-bi-holo-bi-holder-for-quasidiscs} is a direct corollary of the following two theorems and the Riemann mapping theorem.
\begin{theorem}[{\cite[p. 94]{Po}}]\label{thm-quasisymmetryextension}
Any conformal map from the unit disc onto a bounded quasidisc extends to a quasisymmetric map from $\bC$ to itself.  
\end{theorem}

\begin{theorem}[Mori's theorem {\cite[Ch. 3, Section C]{ahlfors}}]\label{thm-mori}
Quasisymmetric maps from $\bC$ to itself are bi-H\"older.
\end{theorem}

In what follows, we present background that will only be used in Sections \ref{chapter-complements-to-a-quasiarc} and \ref{section-koebeschwartz}.
A quasisymmetric (or quasiconformal) map, $\phi\colon\bC\to\bC$ always satisfies the Beltrami equation (weakly)
\begin{align}\label{eq-beltrami}
    \frac{\partial \phi}{\partial \bar z} = \mu \frac{\partial \phi}{\partial z},
\end{align}
where the \textit{Beltrami coefficient} $\mu \colon \bC \to \bC$ is a measurable function that satisfies $\|\mu\|_\infty < 1$.
If $\mu = 0$ almost everywhere, then $\phi$ satisfies the Cauchy-Riemann equations weakly and Weyl's lemma implies that $\phi$ is holomorphic.

The following theorem states that given such a $\mu$, one can always solve the Beltrami equation.
\begin{theorem}[Measurable Riemann Mapping Theorem {\cite[Chapter V]{ahlfors}}]\label{thm-measurableriemann}
Let $\mu \colon \bC \to \bC$ be a measurable function such that $\|\mu\|_\infty < 1$.  Then there exists a quasisymmetry $\phi \colon \bC \to \bC$ that satisfies  \eqref{eq-beltrami}.  Additionally, if another map $\psi$ satisfies \eqref{eq-beltrami} for the same $\mu$, then $\phi \circ \psi^{-1}$ is a M\"obius transformation of $\bC$, i.e., an affine linear map.
\end{theorem}

A version of Theorem \ref{thm-quasisymmetryextension} holds for non-simply-connected domains.
\begin{theorem}[{\cite[Ch. II, Theorem 8.3]{lehtovirtanen}}]\label{thm-qsextensionmultiplyconnected}
Let $U, V \subset \bC$ be domains with finitely many connected components of their boundaries.  Suppose that each connected component of the boundaries of $U$ and $V$ is a quasicircle, that at most one connected component
of the boundary of $U$ is unbounded and at most one connected component of the boundary of $V$ is unbounded.
If $\phi \colon U \to V$ is a conformal map, then $\phi$ extends to a quasisymmetry from $\bC$ to itself.
\end{theorem}
We also will use a result that states that countably many points are removable for conformal maps. For a discussion on removability theorems for conformal maps see \cite[Ch. V.3]{lehtovirtanen}.
\begin{prop}\label{prop-points-are-removable}
Let $U$ be a domain in $\bC$ and $K \subset U$ a countable set of points that is closed in $\bC$.  If $\phi \colon U\setminus K \to \bC$ is a bounded conformal map into $\bC$, then $\phi$ extends to a conformal map on $U$.
\end{prop}
\begin{proof}
This follows from the results in \cite{ahlforsbeurling}, particularly Theorem 3.
\end{proof}
\begin{corollary}\label{cor-points-are-removable}
Let $U$ be a domain in $\bC$ and $K \subset U$ a countable set of points so that $U\setminus K$ is a domain.  If $\phi \colon U\setminus K \to \bC$ is a bounded conformal map into $\bC$, then $\phi$ extends to a conformal map on $U$.
\end{corollary}
\begin{proof}
Fix $\epsilon > 0$ and let $N(\partial U,\epsilon)$ be an open $\epsilon$-neighborhood of $\partial U$. We decompose $K$ into $K_{\epsilon,1} \cup K_{\epsilon, 2}$, where $K_{\epsilon,1} =K\cap N(\partial U,\epsilon)$ and $K_{\epsilon,2} = K \setminus K_{\epsilon,1}$.  Since $K$ is a countable set of points we can always choose $\epsilon$ arbitrary small and so that $\partial N(\partial U,\epsilon) \cap K = \emptyset$.
With this choice, we have that $K_{\epsilon,1}$ is closed in $K$.  Indeed, if $\{z_i\}_{i \in \mathbb N}$ is a sequence of points in $K_{\epsilon,1}$ that converges to $z_0 \in K$, then $z_0 \in \overline{N(\partial U,\epsilon)}$.  By our choice of $\epsilon$, $z_0 \in N(\partial U,\epsilon)$ and $K_{\epsilon,1}$ is closed in $K$.  This implies that $K_{\epsilon,1}$ is also closed in $U$.

Additionally, $K_{\epsilon,2}$ is closed in $\bC$.  To see this we note that $K_{\epsilon,2} \subset U \setminus N(\partial U,\epsilon)$.  If $\{z_i\}_{i \in \mathbb N}$ is a sequence of points in $K_{\epsilon,2}$ that converges to $z_0 \in \mathbb C$, then $z_0 \in U$.  By the hypothesis that $U\setminus K$ is a domain we have that $z_0 \in K$.
We conclude by Proposition \ref{prop-points-are-removable} that the map $\phi$ extends to a conformal map from $U \setminus K_{\epsilon,1}$ into $\bC$.

Let $\{\epsilon_n\}_{n \in \mathbb N}$ be a sequence converging to $0$ so that $\partial N(\partial U,\epsilon_n) \cap K = \emptyset$ for any $n\in\mathbb{N}$.  This gives a decomposition, $K = K_{\epsilon_n,1} \cup K_{\epsilon_n,2}$, so that $\phi$ extends to a conformal map on $U \setminus K_{\epsilon_n,1}$ for all $n \in \mathbb N$.  Since $\epsilon_n$ converges to $0$, we have that $K = \cup_{n \in \mathbb N} K_{\epsilon_n ,2}$.  Thus, $\phi$ extends to a conformal map on $U$.
\end{proof}

We also record here Koebe's theorem on multiply-connected domains in $\bC$.
\begin{theorem}[{\cite{Koebe}}]\label{thm-koebe}
If $U \subset \bC$ is a domain with finitely many connected components in its boundary, then there exists a conformal map sending $U$ to a domain whose connected components of its boundary consist of points and circles (i.e., a circle domain).  Moreover, the map is unique up to a M\"obius transformation.
\end{theorem}

Theorem \ref{thm-koebe} is a generalization of the Riemann mapping theorem. He and Schramm \cite{HS93} generalized Theorem \ref{thm-koebe} to domains with countably many boundary components, however we will not use this result. Also, it is interesting to mention that
in contrast with the Riemann mapping theorem, the fundamental group of a non-simply-connected domain does not completely determine it up to a conformal map. This may be seen from the uniqueness statement in Theorem \ref{thm-koebe}. A conformal map between two circle domains must be a M\"obius transformation, and so the conformal equivalence class of a given circle domain is rather small (it is described by 4 complex parameters). For instance, two annuli are conformally equivalent if and only if the ratios of their radii are the same.

\section{Schwartz equivalence of bounded quasidiscs}\label{section-on-bounded-quasidiscs}
We start by showing that we can bound the distortion of the ``distance to the boundary function" caused by a H\"older map.
\begin{lemma}\label{lemma-distortion-of-distance}

Let $\phi\colon U\to V$ be a H\"older map between proper open subsets
of $\bR^n$. Then, there exist $\alpha,C_\alpha\in\bR_{>0}$
such that for all $x\in U$, 
\[
\operatorname{dist}(\phi(x),\bR^n\setminus
V)\leq C_\alpha \operatorname{dist}(x,\bR^n\setminus U)^\alpha.
\]

\end{lemma}

\begin{proof}

There exist $C,\alpha\in\bR_{>0}$
such that $\abs{\phi(x)-\phi(y)}\leq C\abs{x-y}^\alpha$ for all $x,y\in U$.
Fix $x_0\in U$ and choose a sequence of points $\{x_i\}_{i\in\mathbb{N}}\subset
U$ converging to a point in $\partial U$, such that 
\begin{align*}
    \abs{x_{i}-x_{i-1}}\le\frac{\operatorname{dist}(x_0,\bR^n\setminus U)}{2^i},
\end{align*}
for any $i\in \mathbb{N}$. 
The sequence of points $\{\phi(x_i)\}_{i\in \mathbb{N}}$ escapes any compact in $V$ and so 
\begin{align*}
    \operatorname{dist}(\phi(x_0),\bR^n\setminus
V)&\le\sum\limits_{i=1}^{\infty}\abs{\phi(x_{i})-\phi(x_{i-1})}\\
&\leq\sum\limits_{i=1}^{\infty}C
\abs{x_i-x_{i-1}}^\alpha \\
&\le\sum\limits_{i=1}^{\infty}C \biggl(\frac{\operatorname{dist}(x_0,\bR^n\setminus
U)}{2^i}\biggr)^\alpha \\
&=C_\alpha  \operatorname{dist}(x_0,\bR^n\setminus U)^\alpha,
\end{align*}
where $C_\alpha:=C\sum\limits_{i=1}^{\infty} (2^{-\alpha})^i> 0$.
\end{proof}

\begin{lemma}\label{lemma-on-fast-decaying}

Let $\phi\colon U\to V$ be a H\"older map between proper open subsets
of $\bR^n$ and let  $s\in\mathcal{S}(V)$. Then, for any $m\in\mathbb{N}_0$
$$\sup\limits_{x\in U}\abs{\frac{s(\phi(x))}{\operatorname{dist}(x,\bR^n\setminus U)^m}}<\infty.$$

\end{lemma}

\begin{proof}

By Lemma \ref{lemma-distortion-of-distance}, there exist $\alpha,C_\alpha\in\bR_{>0}$
such that for all $x\in U$ $$\operatorname{dist}(x,\bR^n\setminus U)\geq(C_\alpha)^{-1/\alpha}(\operatorname{dist}(\phi(x),\bR^n\setminus
V))^{1/\alpha}.$$

Thus, for any $m\in\mathbb{N}_0$ we have that
\begin{align*}
\sup\limits_{x\in U}\abs{\frac{s(\phi(x))}{\operatorname{dist}(x,\bR^n\setminus
U)^m}}&\leq\sup\limits_{x\in U}\abs{\frac{(C_\alpha)^{m/\alpha}  s(\phi(x))}{\operatorname{dist}(\phi(x),\bR^n\setminus
V)^{m/\alpha}}}\\
&=(C_\alpha)^{m/\alpha} \sup\limits_{y\in V}\abs{\frac{s(y)}{\operatorname{dist}(y,\bR^n\setminus
V)^{m/\alpha}}}<\infty,
\end{align*}
where the last inequality follows from Proposition \ref{prop-characterization-of-schwartz-functions}.
\end{proof}

We now switch settings to simply-connected domains in $\bC$.
\begin{lemma}\label{lemma-on-bounded-derivatives}
Let $\phi\colon U\to V$ be a conformal map between open subsets in $\bC$ and assume $V$ is bounded.
Then, for any $n\in\mathbb{N}_0$, there exists $C_n > 0$ depending on $V$ such that for any $z\in U$,
\[
\abs{\phi^{(n)}(z)}\leq
\frac{C_n}{\operatorname{dist}(z, \partial U)^{n}},
\]
where $\phi^{(n)}=(\frac{\partial}{\partial z})^n\phi$.
\end{lemma}

\begin{proof}
Fix $z\in U$. Set the curve $\gamma\subset U$ to be the circle of radius $\frac{\operatorname{dist}(z,\partial U)}{2}$ centered at $z$. Using Cauchy's integral formula, for any $n\in\mathbb{N}_0$ we have
\begin{align*}
    \abs{\phi^{(n)}(z)}&=\abs{\frac{n!}{2\pi i}\oint_\gamma \frac{\phi(w)}{(w-z)^{n+1}}dw}\leq \frac{n!}{2\pi }\oint_\gamma \frac{\abs{\phi(w)}}{\abs{w-z}^{n+1}}dw\\
    &= \frac{n!}{2\pi }\oint_\gamma \frac{2^{n+1}\abs{\phi(w)}}{\operatorname{dist}(z,\partial U)^{n+1}}dw\leq 
\frac{C_n}{\operatorname{dist}(z,\partial U)^{n}},
\end{align*}
 where in the last inequality we used the facts that $V$ is bounded (and so $\abs{\phi(w)}$ is bounded) and that the length of $\gamma$ is $\pi\operatorname{dist}(z,\partial
U)$.
\end{proof}

\begin{theorem}\label{conformalbiholderbetweenbounded}Let $\phi\colon U\to V$ be a conformal
bi-H\"older map between bounded domains in $\bC$. Then, $\mathcal{S}(U)\cong\mathcal{S}(V)$. 
\end{theorem}
\begin{proof}
Recall that for holomorphic
maps, $\|D\phi(z)\| = |\phi'(z)|$, where $D\phi$ is the differential of $\phi$
and $||\cdot||$ is the usual operator norm. Thus, Lemma \ref{lemma-on-bounded-derivatives} implies that for any multi index $k\in(\mathbb{N}_0)^2$,
there exists $C_k\in\bR_{>0}$ such that for any $z\in U$,
\begin{align}\label{eq-derivbound}
    \abs{\phi^{(k)}(z)}\leq
\frac{C_k}{\operatorname{dist}(z,\partial U)^{|k|}}
\end{align}
and for any $z'\in V$,
\begin{align}\label{eq-derivboundinv}
    \abs{(\phi^{-1})^{(k)}(z')}\leq
\frac{C_k}{\operatorname{dist}(z',\partial V)^{|k|}}.
\end{align} 

By Lemma \ref{lemma-enough-to-pull-functions-to-get-F-iso}, it is enough
to show that $s\in\mathcal{S}(V)$ implies $s\circ\phi\in\mathcal{S}(U)$
and that $\tilde s\in\mathcal{S}(U)$ implies $\tilde s\circ(\phi^{-1})\in\mathcal{S}(V)$.
We only show the first part. The second follows in the exact same way. 

Fix some $s\in\mathcal{S}(V)$.
As $U$ is bounded, it is enough to check that all the partial derivatives
of $s\circ\phi$ tend to zero when $z$ approaches the boundary of $U$.
Set $k\in(\mathbb{N}_0)^2$. Using the chain rule and the Leibniz rule, one
easily sees that the partial derivative $(s\circ\phi)^{(k)}$ is a finite
linear combination (with constant coefficients) of terms of the form $(s^{(l)}\circ\phi)\phi^{(l_{1})}_{m_1}\cdots\phi^{(l_t)}_{m_t}$,
where $l,l_1,\dots, l_t\in(\mathbb{N}_0)^2$ are such that $|l|,|l_1|,\dots,|l_t|\leq|k|$
and $m_1,\dots,m_t \in \{1,2\}$.
Recall that $\phi_{m_i}^{(l_i)}$ is the partial derivative corresponding
to the multi-index $l_i$ of the coordinate function $\phi_{m_i}$.
Thus, it is enough to show that every such a term tends to zero as $z$ approaches
the boundary of $U$. Indeed, 
\begin{align*}
\sup\limits_{z\in U}\abs{\frac{(s^{(l)}\circ\phi(z))\phi^{(l_{1})}_{m_1}(z)\cdots\phi^{(l_t)}_{m_t}(z)}{\operatorname{dist}(z,\bC\setminus
U)}}\leq\sup\limits_{z\in U}\abs{\frac{(\prod\limits_{i=1}^{t} C_{l_i})(s^{(l)}(\phi(z)))}{\operatorname{dist}(z,\bC\setminus
U)^{\sum\limits_{i=1}^{t}|l_i|+1}}}<\infty,
\end{align*}
where the first inequality follows from $\eqref{eq-derivbound}$ and the second
follows from the fact that $s^{(l)}\in\mathcal{S}(V)$ and from Lemma \ref{lemma-on-fast-decaying}.
\end{proof}

\begin{corollary}\label{main-theorem-forbounded-quasidiscs}
Any two bounded quasidiscs are Schwartz equivalent.
\end{corollary}

\begin{proof}
It is enough to prove that any bounded quasidisc is Schwartz equivalent to the unit disc. This follows immediately from Proposition \ref{prop-on-bi-holo-bi-holder-for-quasidiscs} and Theorem \ref{conformalbiholderbetweenbounded}.\end{proof}

\section{Unbounded domains whose boundaries are quasicircles}\label{section-on-unbounded-quasidiscs}
Having shown Corollary \ref{main-theorem-forbounded-quasidiscs} we now proceed to show the case when the domains in question are not necessarily bounded.
\begin{lemma}\label{lemma-mobiuspreservesschwartz}
If $U,V \subset \bC$ are domains such that there exists a M\"obius transformation from $U$ onto $V$, then $\mathcal{S}(U) \cong \mathcal{S}(V)$. 
\end{lemma}
\begin{proof}
The group of M\"obius transformations is generated by the maps $z \mapsto az + b$, for $a,b \in \bC$ and $z \mapsto \frac{1}{z}$.
It is clear that maps of the first type preserve the Schwartz space.  So it suffices to show that $A(z) = \frac{1}{z}$ preserves the Schwartz space.
By Lemma \ref{lemma-enough-to-pull-functions-to-get-F-iso}, it is enough to check that $s\in\mathcal{S}(U)$ implies $s\circ A^{-1}\in\mathcal{S}(A(U))$ and $s\in\mathcal{S}(A(U))$ implies $s\circ A\in\mathcal{S}(U)$.  Since $A^{-1} = A$, we only need to show one case.

Fix $s\in\mathcal{S}(A(U))$. 
If $A(U)$ is bounded, or equivalently if $U$ is bounded away from the origin, then $A$ has bounded derivatives of all orders and $s \circ A \in \mathcal{S}(U)$.
Suppose $A(U)$ is unbounded.
It is enough to show that all the partial derivatives of $s\circ A$ tend to zero as $z$ approaches any point in $\partial U$ and that they have Schwartz-like decay at $\infty$ (in the case when
$U$ is unbounded).
Since $U$ and $A(U)$ are subsets of $\bC$, we have that $0$ and $\infty$ are in $\bC \setminus U$ and $\widehat\bC \setminus A(U)$, respectively.

Set $z_0\in\partial U$ and let $\{z_1,z_2,\dots\}$ be a sequence of points in $U$ converging to $z_0$.
For $i\in\mathbb{N}$ set $y_i:=A(z_i)$. The sequence $\{y_1,y_2\dots\}$ converges to $y_0$ (note that $y_0$ may be $\infty$ and that this happens if and only if $z_0=0$).
Now set $k\in(\mathbb{N}_0)^2$. As before, using the chain rule and the Leibniz rule, one sees that the partial derivative $(s\circ A)^{(k)}$ is a finite linear combination (with constant coefficients) of terms of the form
$$(s^{(l)}\circ A) A^{(l_{1})}_{m_1}\cdots A^{(l_t)}_{m_t},$$
where $l,l_1,\dots,l_t\in(\mathbb{N}_0)^2$, $1\leq|l|,|l_1|,\dots,|l_t|\leq|k|$ and $m_1,\dots,m_t \in \{1,2\}$.
Thus it is enough to show that the sequence
\begin{align*}
    \{(s^{(l)}\circ A(z_i)) A^{(l_{1})}_{m_1}(z_i)\cdots A^{(l_t)}_{m_t}(z_i)\}_{i \in \mathbb N},
\end{align*}
which can be rewritten as
\begin{align}\label{eq-thesequence}
    \{(s^{(l)}(y_i)) A^{(l_{1})}_{m_1}(z_i)\cdots A^{(l_t)}_{m_t}(z_i)\}_{i \in \mathbb N},
\end{align}
tends to zero. Since $s^{(l)}\in\mathcal{S}(A(U))$,
if $z_0\neq0$ (i.e., $y_0\neq\infty$), then in a neighborhood of $z_0$ all the derivatives
of $A$ are bounded and the sequence \eqref{eq-thesequence} tends to zero. 
Otherwise $z_0=0$ and, for any $|l_j|\geq1$ and for $z$ close enough to $0$,

\begin{align}\label{eq-boundonratinal}
    \text{ }\abs{A^{(l_{j})}(z)}=\frac{|l_j|!}{|z|^{|l_j|+1}}\leq\abs{A(z)}^{|l_j|+2},
\end{align}
since $A(z)=\frac{1}{z}$.
In particular \eqref{eq-boundonratinal} implies that, for sufficiently large $i\in\mathbb{N}$, we have 

\begin{align}\label{eq-boundonratinalfinal}
    \text{ }\abs{A^{(l_{j})}_{m_j}(z_i)}\leq |y_i|^{|l_j|+2}. 
\end{align}
Recall once again that $s^{(l)}\in\mathcal{S}(A(U))$ and we conclude using \eqref{eq-boundonratinalfinal} that \eqref{eq-thesequence} tends to zero. 

To see the Schwartz-like decay at $\infty$ (if $U$ is unbounded), note that
the partial derivatives of $A$ are uniformly bounded away from the origin.
So if $k_1,k_2 \in \mathbb N_0$, then there exists a constant $C >0$ so that
for any $(x,y) \in \bR^2$ and $z$ with $\abs{z}>1$,
\begin{align*}
     |x^{k_1}y^{k_2} s^{(l)}(A(x,y))   A^{(l_1)}_{m_1}(x,y)\cdots A^{(l_t)}_{m_t}(x,y)|
&\le C|x^{k_1}y^{k_2} s^{(l)}( A(x,y))|\\
&\leq C\abs{z}^{k_1+k_2}\abs{s^{(l)}\bigg (\frac{1}{z}\bigg )}.
 \end{align*}
 This goes to $0$ as the origin lies in the boundary of $A(U)$ and $s^{(l)}\in\mathcal{S}(A(U))$.
Thus we have proved that $s\circ A\in\mathcal{S}(U)$.
\end{proof}
\begin{theorem}\label{theorem-quasidiscs-are-schwartz-equivalent}
Any two quasidiscs are Schwartz equivalent.
\end{theorem}

\begin{proof}
Let $U\subset\mathbb{C}$ be an arbitrary quasidisc. By Corollary \ref{main-theorem-forbounded-quasidiscs}, it is enough to find a bounded quasidisc $U'\subset\bC$ such that $U$ and $U'$ are Schwartz equivalent.

By Remark \ref{remark-quasidisc-cannot-be-dense}, there exists a point $z'\in\bC\setminus\overline U$. Let $A(z):=\frac{1}{z-z'}$. The map $A$ is a M\"obius transformation of $\widehat \bC$ such that $A(z')=\infty$. It is a well-known fact that M\"obius transformations map quasidiscs to quasidiscs.
Therefore, the set $A(U)$ is a quasidisc and moreover it is bounded.
By Lemma \ref{lemma-mobiuspreservesschwartz}, $S(U) \cong S(A(U))$, which proves the theorem.\end{proof}

We now show that unbounded domains whose boundaries are bounded quasicircles are all Schwartz equivalent to $\bC \setminus \overline{\mathbb D}$.  This case is very similar to when a domain is a bounded quasidisc.  However, now one must consider the behavior of a Schwartz function composed with a Riemann map near $\infty$.

\begin{theorem}\label{thm-unbouded-domains-whose-boundaries-are-bounded-quasicircles}
If $U$ is an unbounded domain whose boundary is a bounded quasicircle, then $\mathcal{S}(U) \cong \mathcal{S}(\bC \setminus \overline{\mathbb D})$.
\end{theorem}
\begin{proof}
By translating we may assume that $0 \notin \overline{U}$ (see Remark \ref{remark-quasidisc-cannot-be-dense}).
Let $U'$ be the image of $(U\cup\{\infty\})$ by the map $1/z$.  The set $U'$ is bounded and simply-connected.  By the Riemann mapping theorem there exists a conformal map $\psi \colon U' \to \mathbb D$, and we can assume $\psi(0)=0$.
Define the map $\phi \colon U \to \bC \setminus \mathbb D$ as
\begin{align*}
    \phi(z):= 1/\psi(1/z).
\end{align*}
The map $\phi$ is defined so that the following diagram commutes:
\[\begin{tikzcd}
U \cup \{\infty\}\arrow{r}{\phi} \arrow[swap]{d}{1/z} & \widehat{\bC} \setminus \overline{\mathbb D} \arrow{d}{1/z} \\
U' \arrow{r}{\psi} & \mathbb D
\end{tikzcd}
\]
Note that $\phi$ extends to $\infty$ so that $\phi(\infty) = \infty$.  So $\phi$ is a well-defined map from $U$ to $\mathbb C \setminus \overline{\mathbb D}$.
From the definition, we see that $\phi$ is a conformal map.
By Theorem \ref{thm-quasisymmetryextension}, $\psi$ extends to a quasisymmetry from $\bC$ to $\bC$.  Immediate from its definition, $\phi$ can also be extended to a quasisymmetry from $\bC$ to $\bC$.

If we can verify that Equations \eqref{eq-derivbound} and \eqref{eq-derivboundinv} hold (with $V=\mathbb{D}$) for $\phi$ near $\partial U$, then the same proof as in Theorem \ref{conformalbiholderbetweenbounded} shows that for any $s \in \mathcal{S}(U)$, all the partial derivatives of $s \circ \phi^{-1}$ tend to zero  as $z$ approaches any boundary point of $\bC\setminus\overline{\mathbb D}$. Similarly, for any $\tilde{s} \in \mathcal{S}(\bC \setminus
\overline{\mathbb D})$, all the partial derivatives of $s
\circ \phi$ tend to zero as $z$ approaches any boundary
point of $U$.  
So it suffices to verify Equations \eqref{eq-derivbound} and \eqref{eq-derivboundinv} and in addition to show that $s\circ \phi$ and $\tilde{s}\circ \phi^{-1}$ have Schwartz-type decay at $\infty$.

We first show Equation \eqref{eq-derivbound} holds for $z$ close to $\partial U$.  Since $\partial U$ is bounded, its image by $\phi$ is also bounded.  
We choose a radius $R > 0$ so that $\partial U \subset \phi^{-1}(B(0,R))$.
If $z \in U$ is sufficiently close to $\partial U$, then $z$ will be closer to $\partial U$ than to $\phi^{-1}(\partial B(0,R))$.  
$\phi$ is a conformal map from $U \cap \phi^{-1}(B(0,R)) \to B(0,R) \setminus \overline{\mathbb{D}}$. For such a $z$, by Lemma \ref{lemma-on-bounded-derivatives} we have
\begin{align*}
    \abs{\phi^{(n)}(z)} \le \frac{C_n}{\operatorname{dist}(z,\partial U)^n},
\end{align*}
where $C_n$ depends on $\phi$ and $U$, but not $z$.  As long as $z$ is sufficiently close to $\partial U$, Equation \eqref{eq-derivbound} holds.
The proof of Equation \eqref{eq-derivboundinv} is the same but applied to $\phi^{-1}$ as opposed to $\phi$.

Let $s \in \mathcal{S}(\bC \setminus \overline{\mathbb D})$.  We next show that $s\circ \phi$ has Schwartz-type decay at $\infty$. 
The partial derivatives of $s\circ \phi$ can be expressed as linear combinations of
\begin{align*}
    (s^{(l)}\circ \phi ) \cdot \phi^{(l_1)}_{m_1}\cdots \phi^{(l_t)}_{m_t},
\end{align*}
where $l,l_1,\dots,l_t \in (\mathbb N_0)^2$ and $m_1,\dots,m_t \in \{1,2\}$.
Since $\phi$ is conformal on $U$, we have that for large $z$,
\begin{align*}
    \phi(z) = a_0 + a_1z + O(|z|^{-1}). 
\end{align*}
Differentiating term by term we see that all the partial derivatives of $\phi$ are bounded for $z \in \bC$, where $|z|$ is large. So, for any $k \in \mathbb N_0$, there exists a constant $C > 0$ so that for any $z \in \bC$,
\begin{align*}
    \sup\limits_{z\in U}\{|z^k  (s^{(l)}\circ \phi(z))\cdot  \phi^{(l_1)}_{m_1}(z)\cdots \phi^{(l_t)}_{m_t}(z)|\} &\le C\sup\limits_{z\in U}\{|z|^k  |s^{(l)}\circ \phi(z)|\} \\
    &=C\sup_{ y\in\bC \setminus \overline{\mathbb D}}\{ |\phi^{-1}(y)|^k  |s^{(l)}(y)|\},
\end{align*}
where $y:=\phi(z)$. We now use the Taylor expansion of $\phi^{-1}$ near $\infty$ to see that $|\phi^{-1}(y)|$ is bounded, up to a constant, by $|y|$.  So,
there exists a constant $C'>0$ such that \begin{align*}
    \sup\limits_{y\in\bC \setminus \overline{\mathbb D}}\{|\phi^{-1}(y)|^k  |s^{(l)}(y)|\} \leq C' \sup\limits_{y\in\bC \setminus \overline{\mathbb D}}\{|y|^k  |s^{(l)}(y)|\}<\infty,
\end{align*}
where the last inequality holds since $s \in \mathcal{S}(\bC \setminus \overline{\mathbb D})$.

The proof for $s\circ \phi^{-1}$, when $s \in \mathcal{S}(U)$, is exactly the same.
\end{proof}

\section{Quasiarc Domains}\label{chapter-complements-to-a-quasiarc}

We can prove a similar statement to Theorem \ref{theorem-quasidiscs-are-schwartz-equivalent} for domains that consist of the plane minus a quasiarc.  
We first need some preliminary statements.  

\begin{lemma}\label{lemma-squareroot}
If $L \subset \bC$ is a quasiarc whose endpoints are $0$ and $\infty$,
then the preimage of $L$ by the map $z\mapsto z^2$ is a quasicircle.
\end{lemma}
\begin{proof}
Let $\gamma$ be the preimage of $L$ by $z^2$.
Note that there exists a  quasisymmetric map $\psi\colon\bC\to\bC$ that sends $\mathbb C \setminus L$ to $\mathbb C \setminus [0,\infty)$.  Suppose there exists a quasisymmetry $\phi \colon \mathbb C \to \mathbb C$ such that $\phi(z)^2 = \psi(z^2)$.  That is, the following diagram commutes:
\[\begin{tikzcd}
\bC \arrow{r}{\phi} \arrow[swap]{d}{z^2} & \bC \arrow{d}{z^2} \\
\bC \arrow{r}{\psi} & \bC
\end{tikzcd}
\]
Then $\gamma = \phi^{-1}(\mathbb R)$ and $\gamma$ is a quasicircle.
Therefore, it suffices to find such a $\phi$.

We do this by employing the measurable Riemann mapping theorem (Theorem \ref{thm-measurableriemann}). For any quasisymmetric map $\rho\colon\bC\to\bC$ we define its Beltrami coefficient $\mu_{\rho}$ by 
\begin{align*}
    \mu_{\rho}(z) := \frac{\partial \rho(z)}{\partial \bar z}\bigg / \frac{\partial \rho(z)}{\partial z}.
\end{align*}

Now, define
\begin{align*}
    \mu(z) := \frac{\partial \psi(z^2)}{\partial \bar z}\bigg / \frac{\partial \psi(z^{2})}{\partial  z} = \mu_\psi(z) \frac{\bar z}{z}.
\end{align*}

Since $\psi$ is a quasisymmetry on $\bC$ and $\abs{\frac{\bar z}{z}}=1$, we have that $\|\mu\|_\infty = \|\mu_{\psi}\|_\infty < 1$.
Theorem \ref{thm-measurableriemann} now implies that there exists a quasisymmetry $\phi\colon\bC\to\bC$ satisfying the Beltrami equation for $\mu$, i.e., such that \eqref{eq-beltrami} holds with $\mu=\mu_\phi$. This $\phi$ is unique up to a M\"obius transformation, and so we fix one such $\phi$ for now.

A straight forward computation shows that the Beltrami coefficient for $\psi \circ p \circ \phi^{-1}$, where $p(z) = z^2$, is
\begin{align*}
    \mu_{\psi \circ p \circ \phi^{-1}} = \frac{\partial \phi/ \partial z}{\overline{\partial \phi/ \partial z}}\frac{\mu_{\psi \circ p} - \mu_\phi}{1- \overline{\mu_\phi}\mu_{\psi \circ p}} = 0,
\end{align*}
where we used the fact that by the definition of $\mu_\phi$, 
\begin{align*}
    \mu_{\psi \circ p} = \mu_\phi.
\end{align*}
So $\psi \circ p \circ \phi^{-1}$ is a holomorphic degree two map from $\mathbb C$ to $\mathbb C$ and therefore is a quadratic polynomial.
If we normalize $\phi$ by a M\"obius transformation so that $\phi(0) = 0$, then $\psi \circ p \circ \phi^{-1}(z) = 0$ only when $z = 0$.
So $\psi \circ p \circ \phi^{-1}(z) = az^2$ and we can change $\phi$ again by a M\"obius transformation so that $\psi \circ p \circ \phi^{-1}(z) = z^2 = p(z)$.  This proves the lemma.
\end{proof}

Let $L \subset \bC$ be a quasiarc whose endpoints are $0$ and $\infty$. 
 The set $U = \bC \setminus L$ is simply-connected and hence there exists
a holomorphic square root $\sqrt{\cdot} \colon U \to \bC$. We now will show that a holomorphic square root preserves Schwartz spaces.
\begin{lemma}\label{lemma-squareroot-preserves-schwartz}
Let $U,U'$ be simply-connected domains in $\bC$ that do not contain $0$. If $g \colon U \to U'$ is a holomorphic square root that maps $U$ onto $U'$, then $\mathcal{S}(U)\cong \mathcal{S}(U')$. 
\end{lemma}
\begin{proof}
By Lemma \ref{lemma-enough-to-pull-functions-to-get-F-iso}, it suffices to show that for $s \in \mathcal{S}(U')$, the function $s\circ g$ is in $\mathcal{S}(U)$ and for $s \in \mathcal{S}(U)$ the function $s \circ g^{-1}$ is in $\mathcal{S}(U')$.
 
Suppose that $s \in \mathcal{S}(U')$. 
In order to show that $s\circ
g\in\mathcal{S}(U)$, it suffices to show that $s\circ g$, and all of its partial derivatives have Schwartz-like decay at $\infty$ (in the case when $U$ is unbounded), and moreover that they all go to zero when approaching any boundary point of $U$.

Note that any partial derivative of $s\circ
g$ can be written as a linear combination with constant coefficients of terms
of the form
 \begin{align}\label{sequencetoshow}
     (s^{(l)}\circ g)  g^{(l_1)}_{m_1}\cdots g^{(l_t)}_{m_t},
 \end{align}
 where $l,l_1,\dots,l_t \in (\mathbb N_0)^2$ and $m_1,\dots,m_t \in \{1,2\}$. Thus, it is enough to show that \eqref{sequencetoshow} has a Schwartz-like decay at $\infty$ (in the case when $U$ is unbounded), and
moreover that \eqref{sequencetoshow} goes to zero when approaching any boundary point of $U$. 

Around any $z_0\in\partial U$ that is not the origin, the function $g$ extends analytically to a neighborhood of $z_0$, and so $g$ and all of its partial derivatives are bounded in a neighborhood of $z_0$. Since $s^{(l)}\in\mathcal{S}(U')$, $s^{(l)}\circ g$ goes to zero as $z$ approaches $z_0$. These two facts together imply that \eqref{sequencetoshow} goes to zero as $z$ approaches $z_0$. 

To see the Schwartz-like decay at $\infty$ (if $U$ is unbounded), note that the partial derivatives of $g$ are uniformly bounded away from the origin.
So if $k_1,k_2 \in \mathbb N_0$, then there exists a constant $C >0$ so that for any $(x,y) \in \bR^2$ and $z$ with $\abs{z}>1$,
 \begin{align*}
     |x^{k_1}y^{k_2} s^{(l)}(g(x,y)) \cdot  g^{(l_1)}_{m_1}(x,y)\cdots g^{(l_t)}_{m_t}(x,y)| \le C|x^{k_1}y^{k_2} s^{(l)}( g(x,y))|.
 \end{align*}
 If $u + iv = g(x,y) = \sqrt{x+iy}$, then
 \begin{align*}
     |x^{k_1}y^{k_2} s^{(l)}(g(x,y))| = |(u^2 - v^2)^{k_1}(2uv)^{k_2} s^{(l)}(u,v)|.
 \end{align*}
 This goes to $0$ as $(x,y)$ goes to $\infty$, since $(u,v)$ also goes to $\infty$ and $s^{(l)}\in\mathcal{S}(U')$.
 
 We are left to check how \eqref{sequencetoshow} behaves as $z$ approaches the origin (in case the origin lies in $\partial U$). Around the origin, \eqref{sequencetoshow} can be bounded by
 \begin{align*}
     |s^{(l)}(g(z)) \cdot g^{(l_1)}_{m_1}(z)\cdots g^{(l_t)}_{m_t}(z)| \leq C|s^{(l)}(g(z))| |z|^{-m}.
 \end{align*}
 where $z = x + iy$, $m \in \mathbb N_0$ is bounded in terms of the indices $\{l_1,\dots,l_t\}$ and $C>0$ is some constant.
 Changing variables as before, let $w = g(z)$ and
 \begin{align*}
     |s^{(l)}(g(z))| |z|^{-m} = |s^{(l)}(w)||w|^{-2m},
 \end{align*}
 This goes to $0$ as $z$ goes to the origin, since $w$ also goes to the origin, which is a boundary point of $U'$, and $s^{(l)} \in \mathcal{S}(U')$.  We have thus shown that $s\circ g \in \mathcal{S}(U)$.
 
Suppose that $s \in \mathcal{S}(U)$. In order to show that $s\circ
g^{-1}\in\mathcal{S}(U')$, it suffices to show that $s\circ g^{-1}$, and all of its
partial derivatives have Schwartz-like decay at $\infty$ (in the case when
$U'$ is unbounded), and moreover that they all go to zero when approaching
any boundary point of $U'$.

Note that any partial derivative of $s\circ
g^{-1}$ can be written as a linear combination with constant coefficients of terms
of the form
 \begin{align}\label{secondsequencetoshow}
     (s^{(l)}\circ g^{-1}) \cdot (g^{-1})^{(l_1)}_{m_1}\cdots (g^{-1})^{(l_t)}_{m_t},
 \end{align}
 where $l,l_1,\dots,l_t \in (\mathbb N_0)^2$ and $m_1,\dots,m_t \in \{1,2\}$.
Thus, it is enough to show that \eqref{secondsequencetoshow} has a Schwartz-like
decay at $\infty$ (in the case when $U'$ is unbounded), and
moreover that \eqref{secondsequencetoshow} goes to zero when approaching any boundary
point of $U'$. 

As the function $g^{-1}(z)=z^2$ is analytic, around any $z_0\in\partial U'$, including  the origin (in case the origin lies in $\partial U'$), it and its partial
derivatives are bounded. Since $s^{(l)}\in\mathcal{S}(U)$,
$s^{(l)}\circ g^{-1}$ goes to zero as $z$ approaches $z_0$. These two facts together
imply that \eqref{secondsequencetoshow} goes to zero as $z$ approaches $z_0$.

To see the Schwartz-like decay at $\infty$ (if $U'$ is unbounded), note that the partial derivatives
of $g^{-1}$ are bounded linearly, and so if $k_1,k_2 \in \mathbb N_0$, then there exists a constant $C_1 > 0$ so that for any $(x,y) \in \bR^2$,
 \begin{align*}
     |x^{k_1}y^{k_2} s^{(l)}( g^{-1}(x,y) )\cdot (g^{-1})^{(l_1)}_{m_1}(x,y)\cdots (g^{-1})^{(l_t)}_{m_t}(x,y)| \le C_1|x^{k_1'}y^{k_2'} s^{(l)}( g^{-1}(x,y))|,
 \end{align*}
 where $k_1'$ and $k_2'$ are bounded in terms of $\{k_1,k_2,l_1,\dots,l_t\}$. Note that there exists a constant $C_2$ so that
 \begin{align*}
     |x^{k_1'}y^{k_2'} s^{(l)}( g^{-1}(x,y))| \le C_2|z|^{k_1'+k_2'}|s^l( g^{-1}(z))|,
\end{align*}
where $z = x+iy$.  So if $w = g^{-1}(z)$, then
\begin{align*}
     |z|^{k_1'+k_2'}|s^{(l)}( g^{-1}(z))| = |w|^{\frac{k_1'+k_2'}{2}}|s^{(l)}(w)|,
\end{align*}
which goes to $0$ as $z$ goes to $\infty$, as then $w$ also goes to $\infty$ and $s^{(l)}\in\mathcal{S}(U)$. We have thus shown that $s\circ g^{-1} \in \mathcal{S}(U')$.
\end{proof}

\begin{theorem}\label{thm-infinite-slit-domains}
If $U \subset \bC$ is a simply-connected domain whose boundary is a quasiarc, then $\mathcal{S}(U) \cong \mathcal{S}(\mathbb D)$.
\end{theorem}
\begin{proof}
There exists a M\"obius transformation that sends the end points of $\partial U$ to $0$ and $\infty$.
By Lemma \ref{lemma-mobiuspreservesschwartz}, this map will preserve the Schwartz space of our set and so without loss of generality we assume that the boundary of $U$ has endpoints at $0$ and $\infty$.
Let $U'$ be the image of $U$ by a holomorphic square root.  By Lemma \ref{lemma-squareroot}, the boundary of $U'$ is
a quasicircle.  In addition, $U$ and hence $U'$ are simply-connected. So by Theorem \ref{theorem-quasidiscs-are-schwartz-equivalent}, $\mathcal{S}(U') \cong \mathcal{S}(\mathbb D)$.  
Finally, by Lemma \ref{lemma-squareroot-preserves-schwartz}, $\mathcal{S}(U) \cong \mathcal{S}(U') \cong \mathcal{S}(\mathbb D)$.
\end{proof}

\begin{remark}As a consequence of Theorem \ref{thm-infinite-slit-domains}, any $C^1$-arc (i.e., a set $\gamma$ such that there exists a $C^1$-smooth immersion from $\bR$ into $\widehat{\bC}$ where its image of $[0,1]$ is $\gamma$) in $\widehat{\bC}$ that has an endpoint at $\infty$ will be the boundary of a domain that is Schwartz equivalent to $\mathbb D$ ($C^1$ arcs are quasiarcs).  Examples of these domains include the sets 
\[
\mathbb \bR^2 \setminus \{(x,y)\in\bR^2| x\geq0 \text{ and } y=f(x)\},
\]
with $f(x)=e^x$, $f(x)=\sin (x)$ and $f(x)=x^\pi$, but not with $f(x)=\sin(x^2)$ as the last set has a boundary that is not $C^1$ at $\infty$ (compare to Example \ref{examples}(6,7)).
Note that $e^x \colon [0,\infty) \to \bR $  is not definable in any polynomially bounded o-minimal structure and $\sin (x)\colon [0,\infty) \to \bR$ is not definable in any o-minimal structure.
\end{remark}

\begin{theorem}\label{thm-slit-domains}
If $U \subset \mathbb C$ is a domain whose boundary is a bounded quasiarc, then $\mathcal{S}(U) \cong \mathcal{S}(\mathbb C\setminus \overline{\mathbb D})$.
\end{theorem}

\begin{proof}
There exists a M\"obius transformation, $A$, mapping $U$ to a set $U'$ so that the end points of $\partial U$ are mapped to $0$ and $\infty$.
If $V = U' \cup A(\infty)$,
then $V$ satisfies the conditions in Theorem \ref{thm-infinite-slit-domains} and so $\mathcal{S}(V) \cong \mathcal{S}(\mathbb D)$.
By Lemma \ref{lemma-mobiuspreservesschwartz}, we have that $\mathcal{S}(U) \cong \mathcal{S}(U')$.
By applying the M\"obius transformation $z\mapsto\frac{1}{z}$ and Lemma \ref{lemma-mobiuspreservesschwartz}, we see that  $\mathcal{S}(\mathbb D \setminus \{0\})\cong\mathcal{S}(\mathbb
C\setminus \overline{\mathbb
D})$. So if $\mathcal{S}(U') \cong \mathcal{S}(\mathbb D \setminus \{0\})$, then the theorem is proved. This will be shown in the following lemma.
\end{proof}
\begin{lemma}
Let $U, V \subset \bC$ be domains so that there exists a diffeomorphism, $\phi$, mapping $U$ to $V$.  If $\phi$ induces an isomorphism of $\mathcal{S}(U)$ and $\mathcal{S}(V)$, then for any $p \in U$, $\mathcal{S}(U\setminus \{p\}) \cong \mathcal{S}(V\setminus \{\phi(p)\})$.
\end{lemma}
\begin{proof}
Let $s \in \mathcal{S}(V\setminus\{\phi(p)\}) \subset \mathcal{S}(V)$.
Then $s \circ \phi \in \mathcal{S}(U)$ and it suffices to show, by Proposition \ref{prop-characterization-of-schwartz-functions} that for any $m\in\mathbb{N}$
\begin{align*}
    \sup_{z\in U\setminus\{p\}} \abs{ \frac{s\circ\phi(z)}{(z-p)^m}} < \infty.
\end{align*}
Clearly, if $z$ is far from $p$ this is satisfied.  The map $\phi$ is smooth at $p$ and so there exists a constant $C > 0$ so that near $\phi(p)$,  
\begin{align*}
   |\phi^{-1}(y) - p| \ge C|y - \phi(p)|.
\end{align*}
This gives that for $z$ near $p$
\begin{align*}
    \abs{ \frac{s\circ\phi(z)}{(z-p)^m}} &= \abs{ \frac{s(y)}{(\phi^{-1}(y)-p)^m}} \\
    &\le   \frac{|s(y)|}{C|y-\phi(p)|^{m}},
\end{align*}
which is bounded independently of $y$, since $s \in \mathcal{S}(V\setminus \{\phi(p)\})$.  
So for any $z \in U\setminus \{p\}$,
\begin{align*}
     \sup_{z\in U\setminus\{p\}} \abs{ \frac{s\circ\phi(z)}{(z-p)^m}} < \infty.
\end{align*}

The proof for the other direction is exactly the same.
\end{proof}

\section{Schwartz-Koebe Theorem for finitely many quasidisc holes}\label{section-koebeschwartz}
\begin{theorem}\label{schwartzkoebe}
Let $U \subset \bC$ be a domain whose boundary consists of connected components that are either points or quasicircles.  Additionally, assume that there are at most countably many points, at most finitely many quasicircles and no more than one quasicircle is unbounded.  Then $\mathcal{S}(U) \cong \mathcal{S}(V)$, where $V$ is a circle domain, i.e., a domain whose connected components of its boundary are all circles or points.
\end{theorem}

\begin{proof}
Let $Q$ be the set of connected components of $\partial U$ that are quasicircles.  If $Q=\emptyset$, then $U$ is a circle domain and there is nothing to prove. Thus we assume that $Q\neq\emptyset$. As $U$ is connected and $\partial U$ contains a quasicircle, $U$ is either contained in a quasidisc or its complement contains a quasidisc. In both cases $U$ is not dense in $\bC$ (see Remark \ref{remark-quasidisc-cannot-be-dense}). We choose a point $z_0\notin\overline U$ and by applying the M\"obius transformation $z\mapsto \frac{1}{z-z_0}$ on $U$ we get a bounded set that is Schwartz equivalent to $U$ (by Lemma \ref{lemma-mobiuspreservesschwartz}).
The bounded image will also satisfy the hypotheses of the theorem since there is at most one unbounded quasicircle in $Q$.  Thus, we may assume that $U$ is bounded.

By Theorem \ref{thm-koebe}, there exists a conformal map $\phi \colon U \to V$, where $V$ is a circle domain. By applying a M\"obius transformation, we may assume that $V$ is bounded.  We first claim that $\phi$ extends to a quasisymmetry from $\bC$ to $\bC$.

Let $\tilde U$ be the union of $U$ with all of its boundary connected components that are points.
By Corollary \ref{cor-points-are-removable}, $\phi$ extends to a conformal map from $\tilde U$ to $\tilde V$, where $\tilde V$ is the union of $V$ with all of its boundary connected components that are points.  The sets $\tilde U$ and $\tilde V$ satisfy the hypotheses of Theorem \ref{thm-qsextensionmultiplyconnected} and therefore $\phi$ extends to a quasisymmetry from $\bC$ to $\bC$. This implies, by Theorem \ref{thm-mori}, that $\phi$ is bi-H\"older. 

We conclude that $\phi$ is a conformal bi-H\"older map between the bounded domains $U$ and $V$, and so $\mathcal{S}(U)\cong\mathcal{S}(V)$ by Theorem \ref{conformalbiholderbetweenbounded}.
\end{proof}

\section{A simply-connected planar domain that is not Schwartz equivalent to the unit disc}\label{nazarov}The following construction gives two simply-connected proper open subsets of the plane that are not Schwartz equivalent.  They are $C^\infty$-diffeomorphic by the Riemann mapping theorem. This construction was relayed to the authors by F. Nazarov through M. Sodin.

Let $\mathbb D$ be the unit disc in the plane and let $$G:=\bigcup\limits_{n\in\mathbb{N}}\big(B(n,\frac{1}{n^2})\cup(n,n+1)\times(-\frac{e^{-e^{n^2}}}{2},\frac{e^{-e^{n^2}}}{2})\big),$$ i.e., the set $G$ consists of the union of discs centered at the natural numbers $n\in\mathbb{N}$ with decreasing radii $\frac{1}{n^2}$, connected with strips of width 1 and height $e^{-e^{n^2}}$ (see Figure 1 below).

    \begin{figure}[h]
    \centering
    \includegraphics[width=8cm]{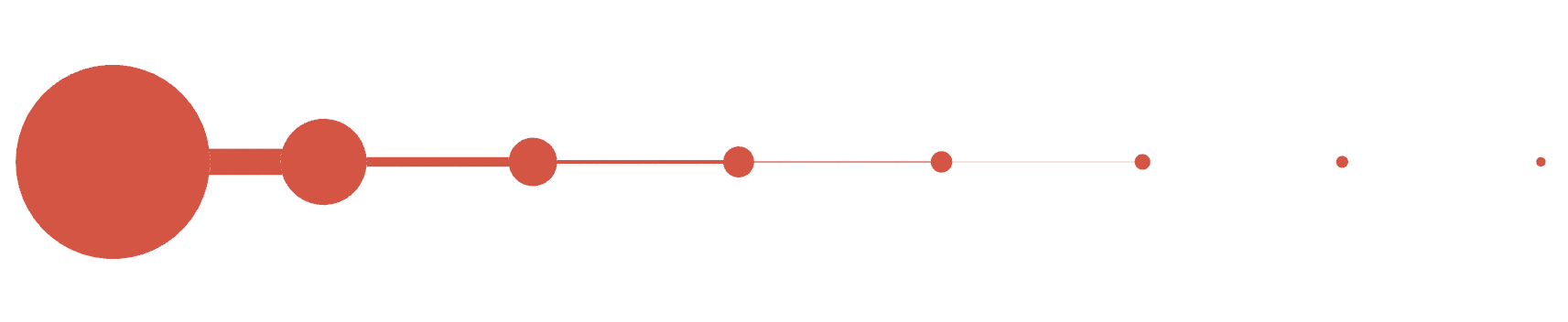}
    \caption{The set $G$}
    \end{figure}
    
We will show that $\mathcal{S}(\mathbb{D})\not\cong\mathcal{S}(G)$. Suppose towards a contradiction that there exists a diffeomorphism $\phi \colon \mathbb{D} \to G$ such that $\phi^{*}|_{\mathcal{S}(G)}:\mathcal{S}(G)\to \mathcal{S}(\mathbb{D})$ is an isomorphism.  Let $K_n := B(n, \frac{1}{2n^2})$
and $D_n := \phi^{-1}(K_n)$.

Let $\chi(z):\mathbb C\to\mathbb{R}$ be a smooth (bump) function such that $\chi$ restricted to the ball of radius $\frac{1}{2}$ around the origin equals 1, and $\chi$ equals 0 outside $\mathbb{D}$. Define \[
f(z):=\sum\limits_{n\in\mathbb{N}}e^{-n}\chi(n^2 (z-n)).
\]
One easily sees that $f\in\mathcal{S}(\bC)$. Moreover, $f$ is  supported in $G$ and is identically $0$ around any boundary point of $G$, and so we constructed  $f \in \mathcal S(G)$ such that $f$ restricted to $K_n$ equals to $e^{-n}$.
    
By the assumption on $\phi$, we have that $f \circ \phi\in\mathcal{S}(\mathbb{D})$, and so there exists $C\in\bR_{>0}$ such that 
\begin{align*}
    \sup_{z \in \mathbb{D}} \frac{|f(\phi(z))|}{d(z,\partial \mathbb{D})} \le C.
\end{align*}
In particular, for any $z\in D_n$ we have $e^{-n}\leq C d(z,\partial\mathbb{D})$, and so 
\begin{align*}
   d(D_n,\partial \mathbb{D})\geq C^{-1} e^{-n}.
\end{align*}
    
Consider the radial Schwartz function $g\in\mathcal{S}(\mathbb{D})$ given by 
\[
    g(z):=e^{\frac{1}{\abs{z}-1}}=e^{\frac{1}{-d(z,\partial \mathbb{D})}}.
\]
In particular, for $z\in D_n$ we have that $g(z)\geq e^{-C e^n}$. For every $n\in\mathbb{N}$, fix some $p_n\in D_n$ and define $\gamma_n$ to be the straight path connecting $p_n$ to $p_{n+1}$. As $g$ is radial we have that for any $n\in\mathbb{N}$, 
\[
    g(z)\geq e^{-C e^{n+1}},
\]
for $z \in \gamma_n$.
Since $\phi$ is a diffeomorphism and $\gamma_n$ is a continuous curve that connects $D_n$ to $D_{n+1}$, we have that $\phi(\gamma_n)$ is a continuous curve that connects $K_n$ to $K_{n+1}$. In particular there exists a point $w_n\in\phi(\gamma_{n})$ with $x$ coordinate $n+\frac{1}{2}$. Note that this point lies in a strip of height $e^{-e^{n^2}}$ in  $G$, and so $d(w_n,\partial G)\leq e^{-e^{n^2}}$. 

By our assumption, $g\circ \phi^{-1}$ is Schwartz and so there exists $C'\in\mathbb{R}_{>0}$ such that 
\[
    \sup_{w \in G} \frac{|g(\phi^{-1}(w))|}{d(w,\partial G)} \le C'
\]
and, in particular, 
\[
    g (\phi^{-1}(w_n)) \leq C'e^{-e^{n^2}}.
\]
On the other hand $\phi^{-1}(w_n)\in\gamma_n$ and so $$g(\phi^{-1}(w_n))\geq e^{-C e^{n+1}},$$which is a contradiction.
    
\begin{remark}The fact that $G$ is unbounded is not essential. There are two easy ways to construct a bounded simply-connected planar domain that is not Schwartz equivalent
to the unit disc using the construction of $G$ above. Firstly, the image of $G$ under the M\"obius transformation $z\mapsto\frac{1}{z-z_0}$ (where $z_0$ is any fixed point outside of $\overline G$) is a such a domain, as by Lemma \ref{lemma-mobiuspreservesschwartz} M\"obius transformations preserve the Schwartz space. Alternatively, one can define $G$ to be a countable union of discs of decreasing radii $\frac{1}{n^2}$, connected
with strips of width $\frac{10}{n^2}$ (instead
of width 1) and height $e^{-e^{n^2}}$. This set is also not Schwartz equivalent to $\mathbb{D}$ by essentially the same arguments.\end{remark}

\section{Planar cusp domains}\label{cusp_domains}As discussed in the Introduction,
any open simply-connected semi-algebraic planar domain is Schwartz equivalent to the unit disc.
In particular, the one sided infinite strip $$S: = \{(x,y) \in \mathbb{R}^2 : 1 < x < \infty, 0 < y < 1\}$$ and the one sided infinite strip with an \emph{algebraic (polynomial)} cusp at infinity $$\{(x,y) \in\mathbb{R}^2 | 1 < x < \infty, 0 < y < \frac{1}{x}\}$$ are such domains. In what follows we will show that the one sided infinite strip with an \emph{exponential} cusp at infinity $$\Omega
:= \{(u,v) \in \mathbb{R}^2 : 0 < u < \infty, 0 < v < e^{-p(u)}\},$$ where $p(u)$ is a non constant polynomial whose highest degree term has a positive coefficient, is Schwartz equivalent to the one sided infinite strip $S$, and by above also to the unit disc.

Indeed, there exists a constant $M > 0$ so that for all $u > M$, $p'(u)$ is positive and moreover $p(M)>1$. If $p$ is of degree 1 then it has the form $p(u)=\mu+\lambda u$ for some $\lambda>0$ and $\mu\in\mathbb{R}$ and in this case we take $M$ such that also $M>\frac{1}{\lambda}$. 
If $p$ is of degree at least 2 we take $M$ such that $p'(u) > 1$ for all $u > M$ and $M>1$. Note that either way $p$ is strictly
increasing whenever $u>M$.  Define the diffeomorphism $\phi \colon \Omega \to S$ as follows.
Let
\begin{align*}
    \phi_1(u,v) = \paren{\frac{e^{p(M)} - 1}{M} u + 1, ve^{p(u)}}
\end{align*}
and
\begin{align*}
    \phi_2(u,v) := (e^{p(u)},ve^{p(u)}).
\end{align*}
Let $\chi \colon \mathbb R \to \mathbb R$ be a  non-decreasing $C^\infty$-smooth function such that $\chi(u) = 1$ when $u \ge 2M$ and $\chi(u) = 0$ when $u \le M$.
Finally, let
\begin{align*}
    \phi(u,v) = (1-\chi(u))\phi_1(u,v) + \chi(u) \phi_2(u,v).
\end{align*}
Let us show that $\phi:\Omega\to S$ is a diffeomorphism. It is clear that $\phi(u,v)$ is smooth on $\Omega$ and one easily see that $\phi(\Omega)=S$.
So it suffices to show that $\phi$ is a homeomorphism.  
Additionally,  whenever $u \le M$, $\phi_1$ is a diffeomorphism and whenever $u \ge 2M$, $\phi_2$ is a diffeomorphism.  So we can consider only the range when $M < u < 2M$.
Since $\phi_1$ and $\phi_2$ agree in the second coordinate and $ve^{p(u)}$ always sends $(0,e^{-p(u)})$ to $(0,1)$ homeomorphically, it suffices to verify that the first coordinate of $\phi$ is a homeomorphism.
This follows if the derivative of $a(u) = (1-\chi(u))\paren{\frac{e^{p(M)} - 1}{M} u + 1} + \chi(u)e^{p(u)}$ is positive.  We calculate that
\begin{align*}
    a'(u) = -\chi'(u)\paren{\frac{e^{p(M)} - 1}{M} u + 1} + (1-\chi(u))\frac{e^{p(M)} - 1}{M} + \chi'(u) e^{p(u)} + \chi(u) e^{p(u)}p'(u).
\end{align*}
Note that $\chi'(u), \chi(u), (1-\chi(u))$ and $p'(u)$ are all non-negative for $M < u < 2M$ and so $a'(u)$ is positive if
\begin{align*}
    b(u) = e^{p(u)} - \paren{\frac{e^{p(M)} - 1}{M} u  + 1} > 0.
\end{align*}

Note that $b(M)=0$, so it is enough to show that for $M<u<2M$
$$b'(u) = e^{p(u)}p'(u) - \frac{e^{p(M)} - 1}{M}$$
is strictly positive. If $p$ is of degree 1 then we have $$b'(u)=\lambda e^{\mu+\lambda u}-\frac{e^{\mu+\lambda u}-1}{M}>(\lambda-\frac{1}{M})e^{\mu+\lambda u}+\frac{1}{M}>0$$ as $M>\frac{1}{\lambda}$. Otherwise, $p$ is of degree at least 2 and we have for $u > M$, $p'(u) > 1$ and $M>1$.  
So 
\begin{align*}
    b'(u) &= e^{p(u)}p'(u) - \frac{e^{p(M)} - 1}{M} \\
    & > e^{p(u)} - \frac{e^{p(M)} - 1}{M} > 0.
\end{align*}

Therefore $b(u)$ is strictly increasing for $M<u<2M$ and hence is positive.
This shows that $\phi$ is a diffeomorphism.

Furthermore, on any bounded set $K \subset \Omega$, the definition for $\phi$ can be extended to a larger domain that includes $K$.  This gives that $\phi$ and its the derivatives on $\{(u,v) \in \mathbb R^2: 0 < u < 3M\} \cap \Omega$ are bounded uniformly.
The same can be said about $\phi^{-1}$ on $\{(x,y) \in \mathbb R^2 : 1 < x < e^{p(3M)}, 0 < y < 1\}$.

By Lemma \ref{lemma-enough-to-pull-functions-to-get-F-iso}, it suffices to show that both $\phi$ and $\phi^{-1}$ preserve
the Schwartz spaces via composition.  
We have that for $x>e^{p(2M)}$ (i.e., for $u > 2M$),
\begin{align*}
    \phi^{-1}(x,y) = \bigg (g(x),\frac{y}{x}\bigg),
\end{align*}
where $g(x) = p^{-1}(\log x)$.  Note that for any $r>0$ there exists $C_r>0$ such that whenever $x>e^{p(2M)}$,
\begin{align}\label{eq-goodboundforg}
    g(x) < C_r x^r.
\end{align}

We also record the following bounds on the
derivatives of $\phi$ and $\phi^{-1}$ whenever $u>2M$, or $x>e^{p(2M)}$ (below $k$ is an arbitrary multi-index):
    \begin{align*}
        |\phi^{(k)} (u,v)| &<e^{p(u)} q_k(u) \\
        |(\phi^{-1})^{(k)}(x,y)|&<C_k,
    \end{align*}
    where $C_k\in\bR$ is a positive constant and $q_k$ is a polynomial.
Note that there exists $C'_k,C''_{k}>0$ such that $\abs{q_{k}(x)}<C'_k+C''_k x^{deg(q_k)}$ for any $x>0$. 
   
Let $f \in \mathcal{S}(S)$.  To show that $f \circ \phi\in\mathcal{S}(\Omega)$,
we bound the term
\begin{align}\label{eq-ratiotoboundstripexample}
        \abs{\frac{f^{(k_1)}(\phi(u,v))|\phi^{(k_2)}(u,v)|^{k_3}}{d((u,v),\partial \Omega)^l}},
\end{align}
    where $k_1,k_2$ are appropriate multi-indices and $k_{3},l \in \mathbb{N}$.  

For $u\le 2M$, $\phi$ and $\phi^{-1}$ are $C^\infty$-smooth and bi-Lipschitz up to the boundary.   In particular $\phi$ is bi-H\"older.
So
\begin{align}\label{eq-boundforfonusmaller2m}
    \sup_{\{(u,v) \in \Omega, 0<u\leq2M\}}   \abs{\frac{f^{(k_1)}(\phi(u,v))|\phi^{(k_2)}(u,v)|^{k_3}}{d((u,v),\partial \Omega)^l}}
\end{align}
is bounded.
To see this, let $\rho\in C^\infty(\mathbb{R})$ be such that $\rho|_{(-\infty,e^{2.4M})}\equiv 1$ and $\rho|_{(e^{2.6M},+\infty)} \equiv 0$. Then, $f(x,y)\rho(x)\in\mathcal{S}(S\cap \{(x,y) \in \mathbb R^2
: 1< x < e^{p(3M)}\})$. Lemma \ref{lemma-on-fast-decaying}, applied to $\phi$ restricted to the set  $S \cap \{(x,y) \in \mathbb R^2 : 1< x < e^{p(3M)}\}$, gives a desired bound for \eqref{eq-boundforfonusmaller2m}.
Hence, it is left to bound \eqref{eq-ratiotoboundstripexample} whenever $u>2M$.

For $u > 2M$, by the bounds on the derivatives of $\phi$,
\begin{align*}
        \abs{\frac{f^{(k_1)}(\phi(u,v))|\phi^{(k_2)}(u,v)|^{k_3}}{d((u,v),\partial
\Omega)^l}} &\leq \abs{\frac{f^{(k_1)}(x,y)e^{k_3p(u)}q_{k_2}(u)^{k_3}}{\min(|v|,|v-e^{-p(u)}|)^l}}
\\
        &=\abs{\frac{f^{(k_1)}(x,y)x^{k_3}q_{k_2}(g(x))^{k_3}}{\min(|\frac{y}{x}|,|\frac{y}{x}-\frac{1}{x}|)^l}}
\\
        &=\abs{ \frac{f^{(k_1)}(x,y)x^{k_3+l}q_{k_2}(g(x))^{k_3}}{\min(|y|,|y-1|)^l}}
\\
        &=\abs{ \frac{f^{(k_1)}(x,y)x^{k_3+l}q_{k_2}(g(x))^{k_3}}{d((x,y),\partial S)^l}}.
\end{align*}
By the bound for $q_k$ and \eqref{eq-goodboundforg},
\begin{align*}
       \abs{ \frac{f^{(k_1)}(x,y)x^{k_3+l}q_{k_2}(g(x))^{k_3}}{d((x,y),\partial S)^l}} & \leq \abs{ \frac{f^{(k_1)}(x,y)x^{k_3+l}(C'_{k_2}+C''_{k_2}
(g(x))^{deg(q_{k_2})})^{k_3}}{d((x,y),\partial S)^l}} \\
        & \leq \abs{ \frac{f^{(k_1)}(x,y)x^{k_3+l}(C'_{k_2}+C''_{k_2}
(C_1 x)^{deg(q_{k_2})})^{k_3}}{d((x,y),\partial S)^l}}.
\end{align*}
The last term is uniformly bounded since $f \in \mathcal{S}(S)$.

We also need to show that $f \circ \phi$ has Schwartz-type decay at $\infty$.
 Let $l_1,l_2 ,k_{3}\in \mathbb{N}$ and $k_1,k_2$ be multi-indices. Then,   \begin{align*}
        \abs{u^{l_1}v^{l_2}f^{(k_1)}(\phi(u,v))|\phi^{(k_2)}(u,v)|^{k_3}} &\leq \abs{u^{l_1}
f^{(k_1)}(x,y) e^{k_3p(u)}q_{k_2}(g(x))^{k_3}} \\
        & = \abs{ g(x)^{l_1}x^{k_3}f^{(k_1)}(x,y)q_{k_2}(g(x))^{k_3}} \\
        & \leq \abs{C_{1}x^{k_{3}+l_1} f^{(k_1)}(x,y)(C'_{k_2}+C''_{k_2}
(C_1 x)^{deg(q_{k_2})})^{k_3}},
    \end{align*}
where the last inequality follows again by the bound for $q_k$ and \eqref{eq-goodboundforg}.  The resulting term is bounded since $f$ is Schwartz.
    
We now show the converse direction.  We denote $\psi = \phi^{-1}$.  Let
$f \in \mathcal S(\Omega)$. It suffices to consider only the case when $x>e^{p(2M)}$ (i.e., $u > 2M$), since the bounded case follows by Lemma \ref{lemma-on-fast-decaying} and the smoothness of $\phi^{-1}$ and its derivatives on bounded domains (as explained above).
Let $k_1,k_2$ be some multi indices, and let $k_{3},l \in \mathbb{N}$. 
Then,    
\begin{align*}
    \abs{\frac{f^{(k_1)}(\psi(x,y))|\psi^{(k_2)}(x,y)|^{k_3}}{d((x,y),\partial S)^l}} & \leq  \frac{C_{k_2}\abs{f^{(k_1)}(u,v)}}{\min(|y|,|1-y|)^l} \\
    &= \frac{C_{k_2}\abs{f^{(k_1)}(u,v)}}{\min(|ve^{p(u)}|,|1-ve^{p(u)}|)^l}\\
    &= \frac{C_{k_2}\abs{f^{(k_1)}(u,v)e^{-lp(u)}}}{\min(|v|,|e^{-p(u)}-v|)^l}\\
    &\le \frac{C_{k_2}\abs{f^{(k_1)}(u,v)}}{d((u,v),\partial \Omega)^l},
\end{align*}
which is bounded since $f$ is Schwartz.  We now consider the decay at
$\infty$.  Let $l_1,l_2 \in \mathbb N$, then
\begin{align*}
    \abs{x^{l_1}y^{l_2}f^{(k_1)}(\psi(x,y))\psi^{(k_2)}(x,y)} & \leq \abs{C_{k_2} e^{l_1p(u)}
f^{(k_1)}(u,v)}.
\end{align*}
Since $f$ is Schwartz, there exists $ C_{k_1}>0$ such that for any $(u,v) \in \Omega$,
\begin{align*}
    \abs{f^{(k_1)}(u,v)} &\leq  C_{k_1}  d((u,v),\partial \Omega)^{l_1 + 1} \\
    & \leq  C_{k_1} e^{-(l_1+1)p(u)}.
\end{align*}
Combining these inequalities,
\begin{align*}
    \abs{x^{l_1}y^{l_2}f^{(k_1)}(\psi(x,y))\psi^{(k_2)}(x,y)} \leq C_{k_2}  C_{k_1} e^{-p(u)},
\end{align*}
which is bounded on $\Omega$.  This completes our proof that $\mathcal S(\Omega) \cong \mathcal S(S)$.

\section{Countable unions of intervals (examples in $\bR^1$)}\label{line_examples}In this section we study the Schwartz equivalence of a few families of subsets of the real line, each of these subsets consists of a disjoint countable union of open intervals. The first two examples deal with sets of the form $N_{a}: = \bigcup\limits_{n
\in \mathbb{N}} (n,n+a(n))$, where $a:\mathbb{N}\to(0,1]$ is some function. The simplest example of such a set is  $N:= \bigcup\limits_{n
\in \mathbb{N}} (n,n+1)$, which is obtained by the constant function $a(n) = 1$. In the third example of this section we will show that $N$ and the the complement in the unit interval to the standard $\frac{1}{3}$-Cantor set are Schwartz equivalent.

\begin{example}\label{example_countable_union_equiv} The first example considers a countable union of intervals whose interval sizes do not shrink faster than polynomial speed.  More precisely, if there exist  $c,p,n_0 > 0$ such that $a(n), b(n) \ge c |n|^{-p}$ for all $n>n_0$, then $\mathcal{S}(N_a)\cong\mathcal{S}(N_b)$.

\

Indeed, let $\phi \colon N_a \to N_b$ map $(n,n+a(n))$ to $(n,n+b(n))$ by
$x \mapsto \frac{b(n)}{a(n)}(x-n) + n$, or explicitly
\[
\phi(x) = \frac{b(\floor
x)}{a(\floor x)}(x -\floor x) + \floor x,\quad \phi^{-1}(y) =\frac{a(\floor
y)}{b(\floor
y)}(y-\floor y) + \floor y.
\]
In particular, \begin{align*}
        \phi'(x) =\frac{b(\floor x)}{a(\floor x)}.
    \end{align*} 
   
Suppose that $ f \in \mathcal S(N_b)$. Let us show that $f\circ \phi
\in \mathcal S(N_a)$.
Clearly $f\circ \phi$ goes to zero with all its derivatives when approaching
any boundary point of $N_a$. We are left to show that $f\circ \phi$ goes
to zero with all its derivatives as $|x|\to\infty$, even after being multiplied
by any polynomial. We start with the following observation: For any $l_1\in\mathbb{N}\cup\{0\}$
and any $l_2,l_3,l_4\in\mathbb{Z}$, 
\begin{align}\label{eq-limoffiszero}
    \lim\limits_{|y|\to\infty}|f^{(l_1)}(y)\cdot
y^{l_2} (a(\floor y))^{l_3} (b(\floor y))^{l_4}|=0.
\end{align}
This follows immediately from the asymptotic bounds 
\begin{align*}
    a(\floor y)&<1 \quad\text{and}\quad
    a(\floor y)^{-1}\leq c^{-1}\floor y^{p}\leq c^{-1}y^{p},\\ 
    b(\floor y)&<1 \quad \text{and} \quad 
    b(\floor y)^{-1}\leq c^{-1}\floor y^{p}\leq c^{-1}y^{p},
\end{align*}
and from the fact that
$f$ is a Schwartz function on $N_b$. As $\phi'$ is locally constant, 
    \begin{align*}
        (f\circ \phi )^{(m)}(x) = f^{(m)}(\phi(x))\frac{b(\floor x)^m}{a(\floor
x)^m},
    \end{align*}
    for any $m\in\mathbb{N}\cup\{0\}$. Set $k\in\mathbb{N}\cup\{0\}$, and
write $y=\phi(x)$. We have $\floor x=\floor y$ and 
\begin{align*}
    |x^k(f\circ \phi) ^{(m)}(x)|&= \frac{|\phi^{-1}(\phi(x))|^k
b(\floor
x)^m|f^{(m)}(y)|}{a(\floor
x)^m} \\
&=\bigg |\frac{a(\floor y)}{b(\floor
y)}(y-\floor y) + \floor y\bigg|^{k}\frac{
b(\floor
y)^m|f^{(m)}(y)|}{a(\floor
y)^m}.
\end{align*}
By \eqref{eq-limoffiszero}, this term foes to $0$ as $|y|\to\infty$.  Hence $f \circ \phi \in \mathcal S(N_a)$.
A symmetric argument
shows that $f\circ \phi^{-1}
\in \mathcal S(N_b)$ for any  $ f \in \mathcal S(N_a)$ and so $\mathcal{S}(N_a)\cong\mathcal{S}(N_b)$ by Lemma \ref{lemma-enough-to-pull-functions-to-get-F-iso}.

\end{example}

\begin{example}\label{example_countable_union_non_equiv}The second example considers the case of a countable union of intervals where the interval sizes shrink faster than any polynomial.  In this setting the union of intervals is not Schwartz equivalent to a countable union of unit intervals.
More precisely, assume $a(n)$
is such that for all $p > 0$ there exists $n_p \in \mathbb{N}$ such that $a(n)
\le n^{-p}$ for any $n \ge n_p$.  Then, 
$N_a$ is not Schwartz equivalent to $N:=\bigcup\limits_{n
\in \mathbb{N}} (n,n+1)$.
The typical example to have in mind is the case when $a(n) = e^{-n}$.

\
    
    Indeed, let $\phi \colon N_a \to N$ be any diffeomorphism. Suppose that $\phi((n,n
+ a(n))) = (m,m+1)$.  Since $\phi$ is a diffeomorphism it induces a bijection
from $\mathbb{N}$ to $\mathbb{N}$, we call the bijection $m(n)$.
There exists an infinite subsequence $n_i \in \mathbb{N}$ so that $m(n_i)
\le n_i$.
    Indeed, assume towards contradiction that there exists $n^*\in\mathbb{N}$
such that $m(n)>n$ for any $n\geq n^*$. In particular $m(n)\neq n^*$ for
any $n\geq n^*$. As $m$ is a bijection on $\mathbb{N}$,  $m$ restricted to
$\{1,2,\dots,n^*-1\}$ is a bijection of $\{1,2,\dots,n^*-1\}$ and $m(n)\neq
n^*$ for any $n\in\{1,2,\dots n^*-1\}$. Thus there does not exist $n\in\mathbb{N}$
such that $m(n)=n^*$, which is a contradiction.

In the sequel let $m_i = m(n_i)$.
We construct $f \in \mathcal S(N)$ such that $f(m_i+\frac{1}{2}) = a(n_i)$.
 Let $\chi(x):\mathbb{R}\to\mathbb{R}$ be a smooth bump function supported
on $[-1/4,1/4]$ such that $\chi(0) = 1$.  
    For any $l\in\mathbb{N}\cup\{0\}$ denote $C_l:=\max\limits_{x\in\mathbb{R}}|\chi^{(l)}(x)|$.
Define
    \begin{align*}
        f(x) := \sum_{i \in \mathbb{N}} a(n_i) \chi(x-m_{i}-1/2).
    \end{align*}
    Clearly $f(x)$ goes to zero with all its derivatives when approaching
any boundary point of $N$ and we only need to verify the decay of $f$ near
$\infty$. Fix $k,l\in\mathbb{N}\cup\{0\}$. For $x \in (m_i,m_i+1)$ we have
\begin{align*}
    |x^kf^{(l)}(x) |&= |x^k a(n_i) \chi^{(l)}(x-m_{i}-1/2)| \\
    &\leq C_{l}|x^k a(n_i)|.
\end{align*}
However, as $m_i\leq n_i$ we also have $x<m_i+1\le n_i+1$ and so
\begin{align*}
    |x^kf^{(l)}(x)| < C_l(n_i+1)^k a(n_i).
\end{align*}
By the assumption on $a(n)$ this goes to $0$ as $n_i\to\infty$. 
Hence, $f \in \mathcal S(N)$.

We next show that $f\circ \phi \notin \mathcal S(N_a)$.
Define $y_i:=\phi^{-1}(m_i+\frac{1}{2})$.
The definition gives that $y_i\in(n_i,n_i+a(n_i)) $, $d(y_i,\partial N_a)<a(n_{i})$ and
\begin{align*}
    \frac{f \circ \phi (y_i)}{d(y_i,\partial N_a)^2}  &= \frac{a(n_i)}{d(y_i,\partial
N_a)^2} \\
    & \ge \frac{a(n_i)}{a(n_i)^2} \\ &= \frac{1}{a(n_i)} \to \infty
\end{align*}
as $i \to \infty$. However, $y_i\to\infty$ as $i\to\infty$ and we conclude
that $f \circ \phi \notin \mathcal S(N_a)$, i.e., $\phi$ does not induce a Schwartz
equivalence. As $\phi$ was an arbitrary diffeomorphism, we conclude that
 $N_a$ is not Schwartz equivalent to $N$.

\end{example}

\begin{example}\label{example_cantor}
In this final example, we show that the complement of the standard $\frac{1}{3}$-Cantor set is Schwartz equivalent to the countable union of unit intervals.
This provides an example in $\mathbb R$ of two sets whose boundaries have different Hausdorff dimensions that are Schwartz equivalent.
Let $C' = [0,1] \setminus C$, where $C$ is the standard
$\frac{1}{3}$-Cantor set.  Then, $C'$ is Schwartz equivalent to $N = \bigcup\limits_{n\in \mathbb{N}} (n,n+1)$.
    
Indeed, we enumerate the intervals of $C'$ so that $C'=\bigcup\limits_{n
\in \mathbb{N}} I_n$, $|I_n|\geq|I_m|$ whenever $n\leq m$, and the left endpoint
of $I_n$ is smaller then the left endpoint
of $I_m$ if both $|I_n|=|I_m|$ and $n<m$. This enumeration is unique,
for example, $I_1 = (1/3,2/3)$, $I_2 = (1/9,2/9)$ and $I_3 = (7/9,8/9)$.
In particular, there are $2^n$ intervals of length $3^{-(n+1)}$.

Let $a(n) := |I_n|$, and so $a(n) = 3^{-m}$ for some $m \in \mathbb N$.
Since there are $2^{m-1}$ intervals of length $3^{-m}$, we have that
$2^{m-1} \le n < 2^{m}$ and so
\begin{align*}
    m-1 \le \frac{\log n}{\log 2} < m,
\end{align*}
and
\begin{align*}
       3^{m} \le 3^{\frac{\log n}{\log 2}+1} =3 n^{\log 3/\log 2} <3 n^{2}.
    \end{align*}

Define $\phi \colon N \to C'$ as the linear map that sends $(n,n+1)$ to $I_n$.
We will show that $\phi$ induces a Schwartz equivalence between $C'$ and $N$.

Fix some $f \in \mathcal S(C')$. We first prove that $f\circ\phi\in\mathcal{S}(N)$.
Clearly $f\circ\phi$ goes to zero with all its derivatives when approaching
any boundary point of $N$ and we only need to verify the decay of $f\circ\phi$
and its derivatives at $\infty$.  Fix $k,l\in\mathbb{N}\cup\{0\}$. 
    We can express $I_n$ as $I_n = (a_n,a_n + 3^{-m})$
where $a_n$ is the left endpoint of $I_n$.
In this notation, $\phi(x) = 3^{-m}(x-n) +a_n$ for any $x \in (n,n+1)$. In particular,
$\phi'(x)=a(n)=3^{-m}$ for any $x \in (n,n+1)$.  Let $\phi(x) =
y$. Then $x=\phi^{-1}(y)=3^m(y-a_n)+n$ for any $y\in I_n$.
So
\begin{align*}
        x^l|(f\circ \phi)^{(k)}(x)| &=  3^{-mk}x^l |f^{(k)}(\phi(x))|\\
        &= 3^{-mk}(3^{m}(y-a_n) +n)^l|f^{(k)}(y)|.
    \end{align*}
    Recall $y \in (a_n,a_n + 3^{-m})$ and so $y-a_n<3^{-m}$ and  $d(y,\partial
C') \le 3^{-m}$.  Additionally, $n<2^m<3^{m}$ and so
    \begin{align*}
        x^l|(f\circ \phi)^{(k)}(x)| &\le 3^{-mk}(1+n)^l|f^{(k)}(y)| \\
        & <3^{-mk}(1+3^m)^lf^{(k)}(y)\\
     & \leq (1+3^m)^lf^{(k)}(y)\\
          & \leq (1+d(y,\partial
C')^{-1})^lf^{(k)}(y).
    \end{align*}
This term tends to $0$ as $y\to\partial C'$ since $f\in\mathcal{S}(C')$.

Now let $f \in \mathcal S(N)$ and let $\psi = \phi^{-1}$.  
We next prove that $f\circ\psi\in\mathcal{S}(C')$. As $C'$ is bounded it is enough to show
that for any $k,l\in\mathbb{N}\cup\{0\}$: $$\sup\limits_{y\in C'}\frac{|(f\circ
\psi)^{(k)}(y)|}{d(y,\partial C')^l}<\infty.$$ 

Let $x=\phi^{-1}(y)=3^m(y-a_n)+n$ for any $y\in I_n$.
We calculate that
    \begin{align*}
        d(x,\partial N) &= d( 3^m (y-a_n) + n, \partial N) \\
        & = \min(3^m(y-a_n), 1 - 3^m (y-a_n)) \\
        & = 3^m \min(y-a_n, 3^{-m} - (y-a_n)) \\
        &= 3^m d(y,\partial C').
    \end{align*} 

Fix $k,l\in\mathbb{N}\cup\{0\}$. If $y \in
I_n$, then $\psi'(y) = 3^m$, $2^{m-1} \le n < 2^{m}$, $d(\psi(y),\partial
N)= 3^m d(y,\partial C')$, and $3^m<3n^2<3\psi(y)^{2}$.
Combining these,
    \begin{align*}
        \frac{|(f\circ \psi)^{(k)}(y)|}{d(y,\partial C')^l} &= \frac{3^{mk}|f^{(k)}(\psi(y))|}{d(y,\partial
C')^l} \\
        &= \frac{3^{mk+ml}|f^{(k)}((\psi(y))|}{d(\psi(y),\partial N)^l}\\
         &\leq  \frac{(3\psi(y)^{2})^{k+l}|f^{(k)}((\psi(y))|}{d(\psi(y),\partial
N)^l}= \frac{(3x^{2})^{k+l}|f^{(k)}(x)|}{d(x,\partial
N)^l},
    \end{align*}
which is bounded since $f\in\mathcal{S}(N)$.

\end{example}

\end{document}